\documentclass[11pt, reqno]{amsart}
\usepackage[latin1]{inputenc}
\usepackage{amssymb}
\usepackage{pdfsync}
\usepackage[english]{babel}
\usepackage{appendix}

\usepackage[pdftex, colorlinks=true,urlcolor=blue,linkcolor=blue,citecolor=blue, linktocpage]{hyperref}
\usepackage{graphicx}

\usepackage[capitalize]{cleveref}
\usepackage{fullpage}
\usepackage{color}
\usepackage{amsmath}
\usepackage{amsfonts}
\usepackage{mathrsfs}
\usepackage{t1enc , graphicx}
\usepackage{verbatim}
\usepackage{bbm}
\usepackage{enumitem}
\usepackage{tikz-cd}
\usepackage{adjustbox}
\usepackage[left= 1.2 in, right= 1.2 in ,top= 1.2 in, bottom = 2.1 in]{geometry}

\newcommand{\R}{\mathbb{R}}

\newcommand{\C}{\mathbb{C}}
\newcommand{\T}{\mathbb{T}}

\newcommand{\fp}{\mathfrak{p}}
\newcommand{\Q}{\mathbb{Q}}
\newcommand{\Z}{\mathbb{Z}}
\newcommand{\N}{\mathbb{N}}
\newcommand{\F}{\mathbb{F}}
\newcommand{\E}{\mathbb{E}}
\renewcommand{\P}{\mathbb{P}}
\newcommand{\Pchen}{\mathbb{P}_{\mathrm{Chen}}}

\newcommand{\abs}[1]{\left\lvert #1 \right\rvert}

\newcommand{\OK}{\mathcal{O}_K}

\newtheorem{theorem}{Theorem}[section]
\newtheorem*{theorem*}{Theorem}
\newtheorem{proposition}[theorem]{Proposition}
\newtheorem{lemma}[theorem]{Lemma}

\newtheorem*{corollary*}{Corollary}
\newtheorem{question}[theorem]{Question}
\newtheorem{conjecture}[theorem]{Conjecture}

\theoremstyle{definition}
\newtheorem*{definition*}{Definition}
\newtheorem{definition}[theorem]{Definition}

\theoremstyle{remark}

\newtheorem*{remark*}{Remark}
\newtheorem{remark}{Remark}

\theoremstyle{plain}
\newcounter{MainTheoremCounter}

\newtheorem{Maintheorem}[MainTheoremCounter]{Theorem}

\theoremstyle{plain}
\newcounter{OldTheoremCounter}

\newcommand{\vertiii}[1]{{\left\vert\kern-0.25ex\left\vert\kern-0.25ex\left\vert #1 
		\right\vert\kern-0.25ex\right\vert\kern-0.25ex\right\vert}}

\author{Pierre-Yves Bienvenu}

\author{John T. Griesmer}    
\author{Anh N. Le}
\author{Th\'ai Ho\`ang L\^e}

\address{Institute of Discrete Mathematics and Geometry, TU Wien,
Wiedner Hauptstr. 8--10, A-1040 Wien, Austria} 
\email{pierre.bienvenu@tuwien.ac.at}

\address{Department of Applied Mathematics and Statistics\\
	Colorado School of Mines\\
	1005 14th Street, Golden, CO 80401, USA} 
\email{jtgriesmer@gmail.com}

\address{Department of Mathematics\\
	University of Denver\\
	2390 S. York St, Denver, CO 80210, USA}
\email{anh.n.le@du.edu}

\address{Department of Mathematics\\
	University of Mississippi\\
	University, MS 38677, USA}
\email{leth@olemiss.edu}

\title{Intersective sets for sparse sets of integers}

\begin{document}
	\maketitle

%\textcolor{red}{Another suggested title: Sets of recurrence for sparse sets of integers.} 

%\textcolor{red}{Another suggested title: Sets of recurrence for sparse sets of integers}
 
	\begin{abstract}
		%We study intersective sets (sets of recurrence) and chromatically intersective sets (sets of chromatic recurrence) for sparse subsets of $\N$. 
  For $E \subset \N$, a subset $R \subset \N$ is \emph{$E$-intersective} if for every $A \subset E$ having positive upper relative density, we have $R \cap (A - A) \neq \varnothing$. On the other hand, $R$ is \emph{chromatically $E$-intersective} if for every finite partition $E=\bigcup_{i=1}^k E_i$, there exists $i$ such that $R\cap (E_i-E_i)\neq\varnothing$. When $E=\N$, we recover the usual notions of intersectivity and chromatic intersectivity. 
  In this article, we investigate to which extent known intersectivity results hold in the relative setting when $E = \P$, the set of primes, or other sparse subsets of $\N$. Among other things, we prove:
  \begin{enumerate}
      \item There exists an intersective set that is not $\P$-intersective.

      \item However, every $\P$-intersective set is intersective. 

      \item There exists a chromatically $\P$-intersective set which is not intersective (and therefore not $\P$-intersective).

      \item The set of shifted Chen primes $\P_{\mathrm{Chen}} + 1$ is $\P$-intersective (and therefore intersective).
  \end{enumerate}
\end{abstract}

\tableofcontents

\section{Introduction}
\subsection{Combinatorial theorems in dense sets of integers and transference to sparse sets}

Let $\N$ be the set of positive integers $\{1, 2, 3, \ldots\}$ and for $N \in \N$, define $[N] = \{1, 2, \ldots, N\}$.
If $A, E \subset \N$, the \emph{upper density of $A$ relative to $E$} is defined as
\[
\overline{d}_{E}(A) := \limsup_{N \to \infty} \frac{|A \cap E \cap [N]|}{|E \cap [N]|}.
\]
Similarly,  the \emph{lower density of $A$ relative to $E$} is 
\[
\underline{d}_{E}(A) := \liminf_{N \to \infty} \frac{|A \cap E \cap [N]|}{|E \cap [N]|}.
\]
When the ambient set $E$ is unambiguous from context, we simply say the upper relative density and the lower relative density of $A$ without mentioning $E$.

Note that $\underline{d}_{E}(A)\leq \overline{d}_{E}(A)$.
If equality holds, we denote by $d_E(A)$ the common value and call it the \emph{density of $A$ relative to $E$}.
If $E=\N$, we omit $E$ and simply write $\overline{d}(A), \underline{d}(A), d(A)$ and call them the upper density, lower density, and density of $A$, respectively. We say a set $A$ of integers is \emph{dense} if $\overline{d}(A)>0$ and \emph{sparse} if $\overline{d}(A) = 0$. More generally, we say that $A$ is \emph{dense relative to $E$} if $\overline{d}_{E}(A)>0$, and that $A$ is \emph{sparse relative to $E$} otherwise.

Dense subsets of $\N$ are known to inherit many combinatorial properties of $\N$. For example, Roth \cite{Roth53} proved that every dense set contains infinitely many three-term arithmetic progressions. Szemer\'edi \cite{Szemeredi75} showed such a set contains arbitrarily long arithmetic progressions. That being said, these properties are not exclusive to dense sets. For instance, despite being a sparse set, the set of primes $\P$ also enjoys the same properties. In \cite{Green_Roththeoremprime}, Green devised a transference principle to deduce from Roth's theorem that every set which is dense relative to $\P$ contains three-term arithmetic progressions. This transference principle was a precursor to another one which enabled Green and Tao \cite{Green_Tao08} to prove that a dense subset of primes contains arbitrarily long arithmetic progressions.
	Since then, many variants of the transference principle were devised
	to prove combinatorial theorems in sparse sets of integers such as the
	squares \cite{broPre}, the sums of two squares \cite{matt}
	and various relatively sparse subsets of the primes.
	
%	The transference principle was developed further in two main directions: to transfer theorems about properties other than
%	arithmetic progressions, including non-linear configurations; and to various other sparse sets, including sparse subsets of the primes. For the first direction, let us mention the ``polynomial Szemer\'edi's theorem'' of Bergelson-Leibman, 
%	transferred to the primes by Tao and Ziegler. In the second direction, an example of relevance for the rest of the paper
%	are Chen primes. A prime $p$ is called a \emph{Chen prime} if $p + 2$ is either a prime or a product of two primes. Let $\Pchen$ denote the set of Chen primes and $[N]$ denote the interval $\{1, 2, \ldots, N\}$. Chen \cite{Chen-OnRepresentation} famously proved that there are infinitely many Chen primes; more specifically, he showed that 
%	for large integer $N$,
%	\[
%	|\Pchen \cap [N]| \geq \frac{c N}{\log^2 N}
%	\]
%	for some absolute constant $c > 0$. And Zhou \cite{kAP-Chen}  proved that any dense subset of Chen primes contains arbitrarily long arithmetic progressions.
	
	Against this backdrop, the first goal of our paper is to investigate whether other combinatorial properties of $\N$ may be transferred to $\P$ and other sparse sets.
	The properties we will study are the so-called intersective properties, which we now define.

\subsection{Intersectivity}
Given $A, B\subset \Z$, we define their \emph{sumset} and \emph{difference set} to be $A + B := \{a + b: a \in A, b \in B\}$ and $A - B := \{a - b: a \in A, b \in B\}$, respectively.
\begin{definition}
    For an infinite set $E \subset \N$, a subset $R \subset \N$ is said to be \emph{$E$-intersective} if for every $A \subset E$ with $\overline{d}_E(A) > 0$, we have $R \cap (A - A) \neq \varnothing$. 
\end{definition}
%(or equivalently, $A\cap (A - R)\neq\varnothing$).
%Recall that $\N$ does not include 0; this avoids making every set that contains 0 intersective.
Thus, $R$ is $E$-intersective if every $A$ satisfying $\overline{d}_E(A) > 0$ contains two distinct elements differing by an element of $R$.
 In the case $E = \P$, we will say $R$ is \emph{prime intersective}. On the other hand, if $E = \N$, then $R$ is simply called \emph{intersective}. 
In the contexts where another notion of intersectivity is also being studied, we will use the terminologies density intersective set and density $E$-intersective set, but for now we omit this precise denomination.

An intersective set is also called a \emph{set of recurrence}. 
It is because by the Furstenberg's Correspondence Principle \cite[Theorem 1.1]{Furstenberg77} and Bergelson's Intersectivity Lemma \cite[Theorem 1.2]{Bergelson-sets-of-recurrence} (see also \cite{Bertrand-Mathis-Ensembles-intersectifs}), a set $R$ is intersective if and only if for every measure preserving system\footnote{A measure preserving system is a quadruple $(X, \mathcal{B}, \mu, T)$ where $(X, \mathcal{B}, \mu)$ is a probability space and $T: X \to X$ is a $\mathcal{B}$-measurable map satisfying $\mu(T^{-1} B) = \mu(B)$ for all $B \in \mathcal{B}$.} $(X, \mathcal{B}, \mu, T)$ and $B \in \mathcal{B}$ satisfying $\mu(B) > 0$, there exists $r \in R$ such that
\begin{equation}\label{eq:set_of_recurrence_def}
	\mu(B \cap T^{-r} B) > 0.
\end{equation}
%Thus we will sometimes use the term ``set of $E$-recurrence'' instead of ``$E$-intersective set.'' 
%Note, however, that in this relative setting, the connection to dynamical systems is lost. 
%\Anote{This raises an interesting question. Do we have a dynamical interpretation of $E$-intersective sets?} 
%\Anote{Do we really use all these terminologies? If not, remove the unused ones.}
	
%The classical Poincar\'e's recurrence theorem says that $\N$ is an intersective set. 
%\Pnote{sorry but I find this sentence stupid; given how we defined intersectivity this is trivial, hence I mentioned Poincar\'e where we make the dynamical connection and will delete this sentence.}
In the late 1970s, S\'ark\"ozy \cite{Sarkozy78} and Furstenberg \cite{Furstenberg77, Furstenberg81} independently proved that $\{n^2: n \in \N\}$ is intersective. Furstenberg used ergodic theory while S\'ark\"ozy's proof is inspired by the original proof of Roth's theorem \cite{Roth53} which employs the circle method. %Kamae and Mend\`es France \cite{???} also came up with another approach shortly after that. 
S\'ark\"ozy \cite{Sarkozy78b} went on and proved that $\{n^2 - 1: n > 1\}$, $\P - 1$ and $\P + 1$ are also intersective, confirming conjectures of Erd\H{o}s. Kamae and Mend\`es France's criterion \cite{Kamae_France78} provides a generalization of S\'ark\"ozy's results to arbitrary polynomials of integer coefficients, which we now state.
	\begin{theorem}[\cite{Kamae_France78}]
		Suppose $Q \in \Z[x]$ has positive leading coefficient. 
		\begin{enumerate}
			\item \label{def:intersective_polynomial} The set $Q(\N)$ is intersective if and only if for every $m \in \N$, there is $n \in \Z$ such that $Q(n) \equiv 0 \pmod{ m}$. If this condition holds, we say $Q$ is an \emph{intersective polynomial}.
			
			\item \label{def:intersective_second_kind} The set $Q(\P)$ is intersective if and only if for every $m \in \N$, there is $n \in \Z$ such that $\gcd(m, n) = 1$ and $Q(n) \equiv 0 \pmod{ m}$.
		\end{enumerate}
	\end{theorem}

All intersective sets mentioned above are also prime intersective. Indeed, using Green's transference principle \cite{Green_Roththeoremprime}, the fourth author \cite{LeThaiHoang} proved that if $Q$ is an intersective polynomial (i.e. satisfies the condition in \eqref{def:intersective_polynomial}), then $Q(\N)$ is prime intersective. 
 This result is now superseded by Rice \cite{rice2}, who proved that if $Q(\N)$ is not $E$-intersective, %\Jnote{This is true, but I was confused by the way this is stated.}
 then $|E \cap [N]| \leq c_1 \frac{N}{(\log N)^{c_2 \log \log \log \log N}}$, where $c_1 >0$ is a constant depending on $Q$ and $c_2 >0$ depends only on the degree of $Q$. Li and Pan \cite{Li-Pan-differencesets} showed that if $Q(1) = 0$ then $Q(\P)$ is also prime intersective. For the general case (when $Q$ satisfies \eqref{def:intersective_second_kind}), $Q(\P)$ is proved to be prime intersective by Rice \cite{Rice_Sarkozytheorem}.

%\Anote{This paragraph tries to convey that the relation between intersective sets and prime intersective is not trivial. Remember that there is a thick set which does not intersect $A - A$ where $\overline{d}_{\P}(A) > 0$. }
%Naturally, a basic necessary condition for $R$ to be $E$-intersective is $R\cap (E-E)\neq \varnothing$. Note that it is widely conjectured that $\P-\P\supset 2\cdot\Z$,
%and this conjecture immediately implies that for any intersective set $R$,
%we have $R\cap (\P-\P)\neq \varnothing$. \Anote{Add more meat into this paragraph if we can or make the connection between the last sentence with the one before that clearer.} \Hnote{or maybe move this to open questions} \Anote{The question has been added in open questions section.}

\subsection{Shifted Chen primes}

%The discussion above hints that all intersective sets are prime intersective. 
Our first result contributes to the previously mentioned pool of sets that are both intersective and prime intersective. The sets that we study come from almost twin primes.
A prime $p$ is a \emph{Chen prime} if $p + 2$ is a product of at most two primes. Chen \cite{Chen-OnRepresentation} proved that the set of Chen primes, denoted by $\P_{\mathrm{Chen}}$, is infinite. Another type of almost twin primes is \emph{bounded gap primes}. For a fixed natural number $h$, let $\P_{\mathrm{bdd}, h}$ be the set of primes $p$ such that $p + h$ is a prime. The celebrated theorem of Zhang \cite{Zhang_boundedgaps} shows that there exists $h \in \N$ such that $\P_{\mathrm{bdd}, h}$ is infinite. 

As previously mentioned, the results of S\'ark\"ozy \cite{Sarkozy78b} and Li-Pan \cite{Li-Pan-differencesets} say that $\P - 1$ and $\P + 1$ are both intersective and prime intersective. Therefore, a natural question is whether $\P_{\mathrm{Chen}} - 1, \P_{\mathrm{Chen}} + 1, \P_{\mathrm{bdd}, h} - 1$ and $\P_{\mathrm{bdd}, h} + 1$ are intersective (and prime intersective) for some $h \in \N$.

An intersective set must contain a nonzero multiple of every natural number. If $p \in \P_{\mathrm{bdd}, h}$ then $p + h \in \P$. Therefore, $\P_{\mathrm{bdd}, h} - 1$ is a subset of $\P - h - 1$ which does not contain a nonzero multiple of $h+1$. Thus $\P_{\mathrm{bdd}, h} - 1$ cannot be intersective. For a similar reason, $\P_{\mathrm{bdd}, h} + 1$ is not intersective unless $h = 2$. This leads to a question: Is $\P_{\mathrm{bdd}, 2} + 1$ intersective?
Even though this question is interesting, it is out of the scope of our investigation because a positive answer will imply that there are infinitely many twin primes.
On the other hand, the matter is more tractable regarding Chen primes, and we are able to prove: 
\begin{Maintheorem}\label{thm:Chen_primes_recurrence}
    $\P_{\mathrm{Chen}} + 1$ is both intersective and prime intersective.
\end{Maintheorem}
To prove \cref{thm:Chen_primes_recurrence}, we use a transference principle developed by the first author, Shao, and Ter{\"a}v{\"a}inen \cite{Bienvenu-Shao-Teravainen}. Due to a local obstruction, we could not use the same method for $\Pchen - 1$ and so the question whether $\P_{\mathrm{Chen}} - 1$ is intersective is still open.

%We prove in fact a much stronger recurrence property, called multiple polynomial recurrence.

\subsection{Separating intersective sets and prime intersective sets}

The similarities in known examples of intersective sets and prime intersective sets raise the question of the existence of intersective sets which are not prime intersective. This question was also asked by the third author in his survey \cite{LeThaiHoang_survey}: 
\begin{question}[{\cite[Problem 6]{LeThaiHoang_survey}}]
	Does there exist an intersective set which is not prime intersective?
\end{question} 

In this paper, we give a positive answer to this question in a rather strong way.
To explain the results in details, we first introduce some important classes of subsets of integers: thick sets and syndetic sets.

%\subsubsection{Thick and syndetic sets, thickly syndetic and piecewise syndetic sets, duality}

%\Anote{In order to keep the introduction appealing and clean, I would trim the next 3 paragraphs and only keep what we need for the introduction of our theorems. We may move some technical definitions to the later sections.} 
%\Pnote{I moved the definition of thickly syndetic a few paragraphs lower. Then the notion of duality is not really required in this paper, so we could suppress it altogether if you feel there are too many definitions.}

A set $S \subset \N$ is \emph{thick} if $S$ contains arbitrarily long intervals of the form $\{m, m + 1, \ldots, m+n\}$. On the other hand, $S$ is \emph{syndetic} if $S - F \supset \N$ for some finite set $F$; equivalently, $S$ is syndetic if the gaps between consecutive elements of $S$ are bounded. 
%The first of class of sets in our study is that of \emph{thick} sets, which are sets that
% contain arbitrarily long intervals of the form $\{m, m + 1, \ldots, n\}$. Another class is that of \emph{syndetic} sets, which are the sets $S\subset \N$ such that there exists a finite set $F\subset\N$ such that $\N\subset S-F$; equivalently, the gaps between consecutive elements of $S$ are bounded.
Note that these two classes of sets are dual to each other in the following sense: given a family $X$ of subsets of $\N$, its dual is $X^*=\{E\subset\N : E\cap A\neq\varnothing \text{ for all } A\in X\}$.
When $X$ is upward closed, i.e. any superset of any member of $X$ is again a member of $X$, it is easy to see that $X^{**}=X$.
Therefore, a set is thick if, and only if, it intersects every syndetic set; and a set is syndetic if, and only if, it intersects every thick set. 
We extend the notions of thick sets and syndetic sets to subsets of $\Z$ in the following way: a set $S \subset \Z$ is thick (or syndetic) if $S \cap \N$ is thick (or syndetic).
%Alternatively speaking, the class of thick sets is the dual of the class of syndetic sets and vice versa.
%Equivalently, a thick set is a set whose complement is not syndetic and a syndetic one is a set whose complement is not thick.

%Note that by definition, the class of intersective sets is the dual of the class of all sets $E\subset \N$ which contains a set of the form $(A-A)\cap \N$, where $A$ has positive upper asymptotic density.
A folklore result says that every thick set $R$ is intersective.
This follows from the fact that any thick set contains an infinite difference set of the form $A-A$, %(see \cite{Furstenberg81a})
\footnote{\label{footnote:thick}We construct the set $A$ inductively. Suppose we have $A_k=\{a_1<\ldots<a_k\}$ such that $(A_k-A_k)\subset R$. Since $R$ is thick, it contains a translate $-A_k$, say $a_{k+1}-A_k$, with $a_{k+1}>a_k$, which completes the induction step.} and any infinite difference set is intersective.\footnote{If a set $X\subset\N$ has positive upper density, the translates $a_i+X$ cannot all be pairwise disjoint, which shows that $(A-A)\cap (X-X)\neq\varnothing$.} 
As a consequence of Pintz's theorem \cite{Pintz}, we know that $\P-\P$ is syndetic.
However, our next result shows that it is no longer the case if $\P$ is replaced by a subset of lower relative density $1$. As a result, we produce a thick set which is intersective but not prime intersective.
%we remove a set of zero density
%, i.e. meets every thick set.
%However, we show that $E-E$ may very well not be syndetic anymore even if $E$ contains
%almost all primes.

\begin{Maintheorem}
\label{th:prime-elem}
There exists $A \subset \P$ such that $d_{\P}(A) = 1$ and $A - A$ is not syndetic. In particular, $R = \N \setminus (A - A)$ is a thick set (and so an intersective set) but not a prime intersective set.

	% There exists a thick set $R$
	% and a subset $A\subset\P$  such that $\underline{d}_{\P}(A) = 1$ and
	% $R\cap (A-A)=\varnothing$. In particular, $R$ is intersective but not prime intersective.
\end{Maintheorem}
	
The main ingredient in the proof of \cref{th:prime-elem} is a classical sieve-theoretic bound for the number
$\pi_m(x)$ of primes $p$ smaller than $x$ such that $p+m$ is also a prime.
 In fact, the theorem applies not only to the primes but to a broad range of sparse sets that satisfy 
similar bounds, for example, the images of non-linear polynomials with integer coefficients and the images of those polynomials evaluated at primes. (See \cref{lm:generalisation}.) %\Anote{Do we have other sets?}
%\Anote{for example, which sets? list them here and in Section 3.}
%on the occurrences of any given difference.

To state our next result, we need to introduce the notions of piecewise syndetic sets and thickly syndetic sets. A set $S \subset \N$ is \emph{piecewise syndetic} if $S = T \cap R$ where $T$ is a thick set and $R$ is syndetic. A set $S \subset \N$ is \emph{thickly syndetic} if $S$ intersects every piecewise syndetic set. Equivalently, $S$ is thickly syndetic if for every $N$, there is a syndetic set $J$ such that $[N] + J \subset S$, i.e. $S$ contains syndetically many copies of intervals of arbitrary length. In particular, a thickly syndetic set is both thick and syndetic. By extension, we say a set $S \subset \Z$ is piecewise syndetic/thickly syndetic if $S \cap \N$ is piecewise syndetic/thickly syndetic.

%A set is \emph{thickly syndetic} if for every sequence $(I_n)$ of intervals of integers of length tending to infinity,
%there is a sequence $(J_n)$ of intervals of length tending to infinity, where $J_n\subset I_n\cap A$ for all $n$.
%Equivalently, a thickly syndetic set intersects every thick set along a thick set.
%It is in particular both thick and syndetic.
%The dual notion is that of a piecewise syndetic set, i.e. a set of the form $T\cap S$, where $T$ is
%thick and $S$ is syndetic.

We currently do not know if the thick set found in \cref{th:prime-elem} can be upgraded to being thickly syndetic.
%The thick set obtained in \cref{th:prime-elem} is explicitly constructed, and in fact is not very special: a wide class of thick sets fail to be prime intersective.
%Yet the subset $A\subset \P$ of density 1, also constructed, depends very much on the thick set.
%In particular, the proof does not indicate that $R$ may be assumed to be thickly syndetic.
However, this upgrade is possible if we slightly weaken the hypothesis on the largeness of $A$.
%at the cost of slightly weaken hypothesis on the largeness of $A$
%\Jnote{Better to state Theorem C in terms of piecewise syndeticity, since Theorem B is stated in terms of syndeticity.  Also better to say that $A$ has relative density $1-\varepsilon$.  I believe (but can't justify) that the difference between upper relative density and relative density in the primes is more significant than the difference between upper density and density in the integers.}
%\Anote{Is it done? As I see now, Theorem C is already stated using piecewise syndetic sets. John: it is.}

\begin{Maintheorem}
\label{thm:pws}
For every $\epsilon>0$, there is a set $A\subset \P$ of relative density at least $1-\epsilon$ such that $\P-A$ is not piecewise syndetic.  In particular, $R:=\N \setminus (\P-A)$ is thickly syndetic, but not prime intersective.
\end{Maintheorem}	

While the proof of \cref{th:prime-elem} uses a quantitative input from number theory regarding the number of bounded gap primes, the proof of \cref{thm:pws} uses a softer approach based on dynamics. More precisely, it uses  minimal rotations on the Bohr compactification $b \Z$ of $\Z$. 

At the cost of replacing ``relative density'' with ``upper relative density,'' \cref{thm:pws} can be extended from $\P$ to any set $E$ whose closure in $b\Z$ has Haar measure zero. This applies to the image of an integer polynomial $\{P(n): n \in \N\}$ where $\deg P \geq 2$, the set of sums of two squares $\{x^2 + y^2: x, y \in \N\}$, and the set of integers represented by a norm form, for example, $\{x^3 + 2y^3 + 4z^3 - 6xyz: x, y, z \in \N\}$.

%The proof relies on the Bohr compactification of the integers; therefore the statement applies
%to any set $E$ whose closure in the Bohr compactification has measure 0.
% Many sets are known to have this property, which is stronger than the property of
% having upper Banach density 0; beyond the set of primes, the set of sums of two squares and sets of the form $P(\Z)$, where $P$ is an integer valued polynomial, do.

% For $E \subset \Z$, its \emph{upper Banach density} is defined by
% \[
%     d^*(E) = \lim_{N \to \infty} \sup_{M \in \Z} \frac{|E \cap [M, M + N)|} {N}.
% \]
% We always have $d^*(E) \geq \overline{d}(E)$. It is easy to see that if $E$ has positive upper Banach density, then $E - E$ is syndetic. 
% Thus with such an $E$ and a subset $A \subset E$ satisfying $\overline{d}_E(A) > 0$, we always have $A - A$ is syndetic.\Jnote{This is false: take $E = \{2^n:n\in \mathbb N\} \cup \bigcup_{n\in \mathbb N} [2^{2^n},2^{2^n}+2^{2^n-1}]$.  Then $\bar{d}(E)=1/3$, and $A:=\{2^n:n\in \mathbb N\}$ has $\overline{d}_E(A)=1$, while $A-A$ is not syndetic.} \Anote{Very good catch! I fixed it below. I also fixed the other instances we mentioned the wrong statement. Previous paragraph will be deleted if we're happy with the next paragraph.}

It is easy to see that if $\overline{d}(A) > 0$, $A - A$ is syndetic. \footnote{Here is a simple proof of this fact: Let $k$ be the largest integer such that there are distinct integers $n_1 < \cdots < n_k$ for which $A + n_i$ are pairwise disjoint. Thus, for any integer $n$, the shift $A + n$ intersects one of the shifts $A + n_i$ and so $n - n_i \in A - A$. It follows that $A - A + \{n_1, \ldots, n_k\} \supset \Z$.} Now if $\underline{d}(E) > 0$ and $A \subset E$ satisfies $\overline{d}_E(A) > 0$, then $\overline{d}(A) > 0$ and so $A - A$ is syndetic.
As a result, it is crucial in Theorems \ref{th:prime-elem} and \ref{thm:pws} that $\underline{d}(\P) = 0$. Given this fact, a natural question is whether having zero lower density is all that is needed to satisfy these two theorems. 
%In particular, Theorems \ref{th:prime-elem} and \ref{thm:pws} are false if one replace $\P$ by a set of positive upper Banach density. This raises a natural question that whether one can replace $\P$ by any set of zero Banach density. 
While we do not know the answer to this question for \cref{thm:pws} (see \cref{conj:RelativeDensityOneMinusEps}), we show in \cref{prop:ss'-z} that merely having zero lower density does not suffice for \cref{th:prime-elem}. In fact, we prove something stronger regarding upper Banach density. For $E \subset \Z$, its \emph{upper Banach density} is defined by
\[
    d^*(E) = \lim_{N \to \infty} \sup_{M \in \N} \frac{|E \cap [M, M + N)|} {N}.
\]
Note that we always have $d^*(E) \geq \overline{d}(E) \geq \underline{d}(E)$. \cref{prop:ss'-z} below states that there exists a set $E$ with $d^*(E) = 0$ such that if $\overline{d}_E(A) = 1$, then $A - A = \Z$ (in particular, syndetic).

%Although the properties obtained here of $\P$ apply to a larger family of
%sparse sets, we show in Section \ref{hypotheses} that they do not apply to every sparse set. \Anote{Note to myself: Add details to this sentence.}

\subsection{Prime intersective sets must be intersective}

Theorems \ref{th:prime-elem} and \ref{thm:pws} say that there exists an intersective set that is not prime intersective. Interestingly, the converse is false, i.e. every prime intersective set must be intersective. Moreover, the same is true if one replace $\P$ by any infinite subset of $\N$ and  this is the content of the next theorem.
%Our next theorem  not only shows that the converse is false regarding the primes but also false for any infinite subset $E \subset \N$.

\begin{Maintheorem}\label{thm:e-intersective-implies-n-intersective}
For any infinite $E \subset \N$, every $E$-intersective set is intersective.
\end{Maintheorem}

\cref{thm:e-intersective-implies-n-intersective} follows from a more general result regarding sets of multiple recurrence (see  \cref{thm:recurrence_with_N}). The idea is to use Furstenberg's Correspondence Principle \cite{Furstenberg77} to recast the problem into a question about of sets of recurrence. We then utilize Fatou's lemma to show that a set which is not a set of recurrence (for $\N$) cannot be a set of recurrence for $E$.
%The idea of their proofs is to use Furstenberg's Correspondence Principle to translate the problem about intersective sets into a question about of sets of recurrence. 
%The idea of \cref{thm:e-intersective-implies-n-intersective}'s proof is as follows: For contradiction, assume $S$ is not an intersective set. By Furstenberg's Correspondence Principle, there are a measure preserving system $(X, \mathcal{B}, \mu, T)$ and a set $A \in \mathcal{B}$ of positive measure such that $\mu(A \cap T^{-n} A) = 0$ for all $n \in S$.  We then use this set $A$ to construct a subset $A' \subset E$ such that  $\overline{d}_E(A') > 0$ and $(A' - A') \cap S = \varnothing$.

\subsection{Chromatic intersectivity versus density intersectivity}

As usual in many additive combinatorial problems, 
besides a density-based notion of intersectivity, there is also a partition-based notion.
\begin{definition}
Given $E\subset\N$, a set $R \subset \N$ is said to be \emph{chromatically $E$-intersective} if for every finite partition $E=\bigcup_{i=1}^k E_i$, there exists $i$ such that $R \cap (E_i-E_i)\neq\varnothing$.
\end{definition}
Equivalently, $R$ is chromatically $E$-intersective if for any finite coloring of $E$, there are distinct $m, n$ of the same color such that $m - n \in R$.

If $E=\N$, we simply say that $R$ is \emph{chromatically intersective}. In dynamical systems language, $R$ is chromatically intersective if and only if for any minimal topological dynamical system $(X, T)$ \footnote{A \emph{topological dynamical system} is a pair $(X, T)$ where $X$ is a compact Hausdorff space and $T: X \to X$ is a continuous map. $(X, T)$ is \emph{minimal} if for every $x \in X$, the orbit $\{T^n x: n \in \N\}$ is dense in $X$.} and any nonempty open set $U \subset X$, there exists $n \in R$ such that $U \cap T^{-n} U \neq \varnothing$.
Due to this characterization, a chromatically intersective set is also called a \emph{set of topological recurrence} (in contrast to measurable recurrence as defined in \eqref{eq:set_of_recurrence_def}).

%Just as the density-based notion of intersectivity may be phrased in terms of measurable
%dynamics, the partition-based notion (when $E=\N$) may be phrased in terms of topological dynamics.
%We omit the details.

In any partition $\N = \bigcup_{i=1}^{\ell} A_i$, one of the $A_i$ has positive upper density, and so an intersective set is always chromatically intersective. Therefore, $\N, \{n^2: n \in \N\}, \{n^2 - 1: n \in \N\}, \P - 1, \P + 1$ are chromatically intersective. On the other hand, K{\v{r}}{\'{\i}}{\v{z}} \cite{Kriz87} proved that there exists a chromatically intersective set which is not intersective.

For a similar reason, for any $E \subset \N$, an $E$-intersective set is chromatically $E$-intersective. From K{\v{r}}{\'{\i}}{\v{z}}'s example, it is natural to ask whether there exists a chromatically $E$-intersective set which is not $E$-intersective. Our next result confirms this is the case. In fact, \cref{thm:chromatic_E_not_density} below strengthens K{\v{r}}{\'{\i}}{\v{z}}'s example by showing that for any subset $E \subset \N$, there exists a chromatically $E$-intersective set which is not intersective. Calling this set $R$, \cref{thm:e-intersective-implies-n-intersective} implies that $R$ is not $E$-intersective and as a result, $R$ is an example of a chromatically $E$-intersective set which is not $E$-intersective.

\begin{Maintheorem}\label{thm:chromatic_E_not_density}
    For any infinite $E \subset \N$, there exists a chromatically $E$-intersective set which is not intersective (and thus not $E$-intersective).
\end{Maintheorem}
The proof of \cref{thm:chromatic_E_not_density} is based on a recent refinement of the third author on K{\v{r}}{\'{\i}}{\v{z}}'s theorem \cite{Kriz87}. 
%\Anote{Rewrite the next two sentences, since the proof is quite constructive.}
Using \cref{thm:chromatic_E_not_density}, we can find a set $R$ that separates chromatic intersectivity and density intersectivity of any infinite $E$. Nevertheless, it is hard to extract from the construction any combinatorial properties of $R$. In contrast, in the special case when $E = \P$, we can take $R$ to be thickly syndetic. Indeed, \cref{thm:pws} says that there is a thickly syndetic set $R$ which is not prime intersective; on the other hand, our next theorem states that every thick set is chromatically prime intersective. Therefore, the thickly syndetic set $R$ found in \cref{thm:pws} is chromatically prime intersective but not prime intersective.

%Regarding the primes, we show that thick sets, which may fail to be prime intersective, are always chromatically prime intersective.
\begin{Maintheorem}
\label{prop:piecewiseSyndetic}
For any finite partition $\P=\bigcup_{i=1}^rE_i$, the union $\bigcup_{i=1}^r (E_i - E_i)$ is syndetic. Equivalently, every thick set is chromatically prime intersective.
\end{Maintheorem}
Since piecewise syndetic sets are partition regular\footnote{That is,
if $k\in\N$ and $A_1,\ldots,A_k\subset\N$ have the property that
$\bigcup_{i=1}^kA_i$ is piecewise syndetic, then one of the $A_i$ is piecewise syndetic.
A proof of this standard fact can be found in \cite[Lemma 1]{brown} or \cite[Theorem 1.24]{Furstenberg81}.}, for any partition $\P=\bigcup_{i=1}^kE_i$, there exists $i\in \{1, \ldots, k\}$ such that $E_i-E_i$ is piecewise syndetic. It remains unclear whether there is $i$ such that $E_i - E_i$ is syndetic.
%It remains unclear whether the word ``piecewise'' may be omitted in the last theorem.
The proof of \cref{prop:piecewiseSyndetic} relies on Maynard and Tao's famous results \cite{Maynard15} on the Hardy-Littlewood prime tuple conjecture.
In fact, Theorem \ref{prop:piecewiseSyndetic} follows from the more general Theorem \ref{th:admissible} below 
which says that for any set $E$ which satisfies some ``finite-tuple'' property, every thick set is chromatically $E$-intersective. This applies to a broad range of sets, such as random sets, various
subsets of the primes, and the set of sums of two squares.

%Note that the existence of chromatically prime intersective, but not density prime intersective set is already shown in \cref{thm:chromatic_E_not_density}. The contribution of \cref{prop:piecewiseSyndetic} is the fact that such a set may be assumed to be thick and its proof is constructive.

One of the crucial 
properties of $\P$ that is used in 
\cref{prop:piecewiseSyndetic} is that there are infinitely many pairs of bounded gap primes. The situation is different for a set whose gaps between consecutive elements tends to infinity. 
%The situation is very different for the set of squares. \Anote{Write for arbitrary set whose gaps go to infinity.}

\begin{Maintheorem}\label{th:squares}
    If $E = \{n_1 < n_2 < \ldots \} \subset \N$ satisfies $\lim_{i \to \infty} (n_{i+1} - n_i) = \infty$, then there is a partition $E = E_1 \cup E_2$ such that $(E_1 - E_1) \cup (E_2 - E_2)$ is not syndetic. In particular, there exists a thick set which is not chromatically $E$-intersective.
\end{Maintheorem}
%Remark that the hypothesis of \cref{th:squares} is equivalent to the fact that there do not exist $M > 0$ and infinitely many pairs $m, n \in E$ such that $|m -n| < M$.

% \begin{Maintheorem}
% \label{th:squares}
% 	There exists  a set that is intersective but not chromatically $E$-intersective
% where $E$ is the set of perfect squares.
% 	More precisely, there exists a thick, hence intersective, set $R$
% 	and a partition $E=A_1\cup A_2$ such that
% 	$R\cap (A_i-A_i)=\varnothing$ for any $i\in\{1,2\}$.
% \end{Maintheorem}
\cref{th:squares} applies, for instance, to $E=\{P(n):n\in\N\}$ where
$P$ is a polynomial of degree at least $2$ and $E = \{ \lfloor n^{1 + \epsilon} \rfloor: n \in \N\}$ for any $\epsilon > 0$. 

\vspace{0.5cm}

\textbf{Diagrams.} Below is a diagram of relations between thick sets, intersective sets, prime intersective sets and their chromatic counterparts. The implications we prove are marked with thick arrows; strike-out arrows mean ``do not imply.''

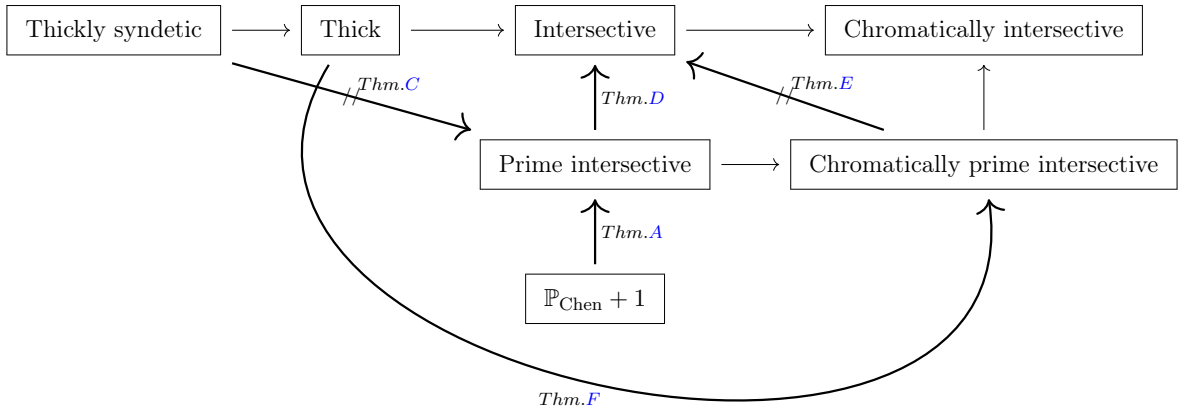
\begin{figure}[ht]
    \centering
\adjustbox{scale=0.85,center}{
    \begin{tikzcd}[row sep= large]
        \fbox{\begin{tabular}{c}
             Thickly syndetic
        \end{tabular}}
        \ar[r,rightarrow] 
        \ar[drr, rightarrow, "//" anchor=center, line width = 1pt, "Thm. \ref{thm:pws}"] & 
        \fbox{\begin{tabular}{c}
             Thick
        \end{tabular}}
        \ar[r, rightarrow] 
        \ar[drr, rightarrow, bend right = 110, looseness=1.4,  line width = 1pt, "Thm. \ref{prop:piecewiseSyndetic}", swap] & 
        \fbox{\begin{tabular}{c}
             Intersective
        \end{tabular}} 
        \ar[r, rightarrow] 
        \ar[d, leftarrow, line width = 1pt, "Thm. \ref{thm:e-intersective-implies-n-intersective}"]& 
        \fbox{\begin{tabular}{c}
            Chromatically intersective
        \end{tabular}} 
        \ar[d, leftarrow]
        \\
        & &
        \fbox{\begin{tabular}{c}
             Prime intersective
        \end{tabular}}
        \ar[r, rightarrow]
        &
        \fbox{\begin{tabular}{c}
        Chromatically prime intersective
        \end{tabular}}
        \ar[ul, rightarrow, "//" anchor=center, line width = 1pt, "Thm. \ref{thm:chromatic_E_not_density}", swap]
        \\
        & & 
        \fbox{\begin{tabular}{c}
             $\P_{\mathrm{Chen}} + 1$
        \end{tabular}}
        \ar[u, rightarrow, line width = 1pt, "Thm. \ref{thm:Chen_primes_recurrence}", swap]
    \end{tikzcd}
    }

    \vspace{-1.3cm}
    \caption{Relations between thick sets, (chromatically) intersective sets, and (chromatically) prime intersective sets.}
    \label{figure:relations_intersective_prime}
\end{figure}

\medskip

The next diagram is the relations between intersective sets and $E$-intersective sets for arbitrary subset $E \subset \N$.

\medskip

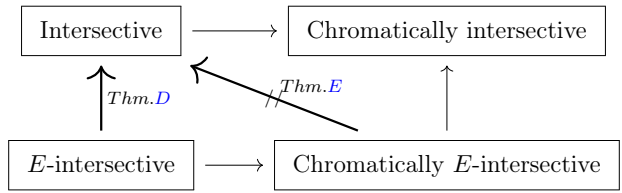
\begin{figure}[ht]
    \centering
\adjustbox{scale=0.85,center}{
    \begin{tikzcd}[row sep= large]
        \fbox{\begin{tabular}{c}
             Intersective
        \end{tabular}} 
        \ar[r, rightarrow] 
        \ar[d, leftarrow, line width = 1pt, "Thm. \ref{thm:e-intersective-implies-n-intersective}"]& 
        \fbox{\begin{tabular}{c}
            Chromatically intersective
        \end{tabular}} 
        \ar[d, leftarrow]
        \\
        \fbox{\begin{tabular}{c}
            $E$-intersective
        \end{tabular}}
        \ar[r, rightarrow]
        &
        \fbox{\begin{tabular}{c}
        Chromatically $E$-intersective
        \end{tabular}}
        \ar[ul, rightarrow, "//" anchor=center, line width = 1pt, "Thm. \ref{thm:chromatic_E_not_density}", swap]
    \end{tikzcd}
    }

    %\vspace{-1.2cm}
    \caption{Relations between (chromatically) intersective sets and (chromatically) $E$-intersective sets for arbitrary $E \subset \N$.}
    \label{figure:relations_intersective_arbitrary_E}
\end{figure}

\medskip

\textbf{Outline of the paper.} 
We prove main theorems in the sequential order: Theorems \ref{thm:Chen_primes_recurrence}, \ref{th:prime-elem}, \ref{thm:pws}, and \ref{thm:e-intersective-implies-n-intersective} are proved in Sections \ref{sec:multiple_recurrence_shifted_checn}, \ref{sec:quantitative-approach},  \ref{sec:proof_of_bohr_compactification}, and \ref{sec:converse-prime-intersective-imply-n-intersective}, respectively. \cref{thm:chromatic_E_not_density} follows from \cref{th:griesmer} which was essentially proved in \cite{Griesmer-separating-topological-recurrence-measurable}. All we have to do is to go through the proof of \cite[Theorem 1.2]{Griesmer-separating-topological-recurrence-measurable} and verify the validity of its relative version. We carry out this tedious task in \cref{sec:appendix_chromatic_not_density}. The rest of \cref{sec:chromatic_vs_density}
is devoted to the proofs of Theorems \ref{prop:piecewiseSyndetic} and \ref{th:squares}. 
%In \cref{sec:concluding_remarks}, we show that \cref{th:prime-elem} is not true if one replace $\P$ by an arbitrary set of zero Banach density zero. 
\cref{sec:open_questions} contains some open questions that natural arise from our study. Lastly, \cref{sec:closure_measure_zero} supplements \cref{thm:pws} by providing a list of subsets of $\N$ whose closures have zero Haar measure in $b\Z$.
\medskip

\textbf{Acknowledgement.}
The first author is supported by FWF Grant I-5554. The third author is supported by an AMS-Simons Travel Grant. The fourth author is supported by NSF Grant DMS-2246921 and a Travel Support for Mathematicians gift from the Simons Foundation.

\section{Shifted Chen primes}

\label{sec:multiple_recurrence_shifted_checn}

\subsection{Preliminaries}

The goal of \cref{sec:multiple_recurrence_shifted_checn} is to prove \cref{thm:Chen_primes_recurrence}. First we recall some terminology and results from a recent paper by Bienvenu, Shao, and Ter{\"a}v{\"a}inen \cite{Bienvenu-Shao-Teravainen}. 

Let $\P_{\mathrm{Chen}}'$ denote the set
\[
    \{p \in \P: p + 2 \text{ is a prime or a product of two primes $p_1 p_2$ with } p_1, p_2 \geq p^{1/10}\}.
\]
Note that $\P_{\mathrm{Chen}}'$ is a subset of $\P_{\mathrm{Chen}} = \{p \in \P: p + 2 \text{ is a prime or a product of two primes}\}$.
We define a weighted indicator function of $\P_{\mathrm{Chen}}'$ to be
\[
    \theta(n) := (\log n)^2 1_{\mathbb{P}'_{\mathrm{Chen}}}(n) 1_{{p|n(n+2)} \implies p \geq n^{1/10}}.
\]
% and of 
% $\mathbb{P}_{\mathcal{H}}$ as
% \[
%     \theta_2(n) := (\log n)^m 1_{\mathbb{P}_{\mathcal{H}}}(n) 1_{p| \prod_{i=1}^m (n + h_i) \implies p \geq n^{\rho}}
% \]
% where $\rho \in (0,1)$. 
%Note that $\P_{\mathrm{Chen}}'$ is only a subset of $\P_{\mathrm{Chen}}$ and sometimes in the literature, the set of Chen primes is defined to be $\P_{\mathrm{Chen}}'$. 
Chen's theorem \cite{Chen-OnRepresentation} says:
\[
    \sum_{n \in [N]} \theta(n) \gg N.
\]
Here, for two functions $f, g: \R_{> 0} \to \R_{>0}$, $f(N) \gg g(N)$ means there exists a positive constant $C$ such that $f(N) \geq C g(N)$ for sufficiently large $N$.
Alternatively, we also use the notation $g(N)\ll f(N)$ for the same meaning.
For any $W \in \N$, the set of Chen primes is heavily biased towards the congruence classes that are coprime to $W$. To remove this local obstruction, we use a standard procedure called ``W-trick''. For any $w \in \N$, let $W = W(w) = \prod_{p \leq w, p \in \P}p$. For $b \in [W]$ such that $(b, W) = 1$, define
\begin{equation}
\label{eq:w-trick-defined}
    \theta_{W, b}(n) = \left( \frac{\phi(W)}{W} \right)^2 \theta(Wn + b).
\end{equation}
Since our focus will be on $\P_{\mathrm{Chen}} + 1$, we will restrict $b$ to $-1$.

For a function $f: S \to \C$ on a finite set $S$, let $\E_{n \in S} f(n)$ denote the average 
\[
    \frac{1}{|S|} \sum_{n \in S} f(n).
\]
For a function $f: \Z_N := \Z/N \Z \to \C$ and $k \in \N$, the \emph{Gowers $U^k$-norm} of $f$ is defined by
\[
    \lVert f \rVert_{U^k(\Z_N)} = \left( \E_{x \in \Z_N} \E_{\underline{h} \in (\Z_N)^k} \prod_{\underline{\omega} \in \{0, 1\}^k} \mathcal{C}^{|\underline{\omega}|} f(x + \underline{\omega} \cdot \underline{h})  \right)^{1/2^k}
\]
where $\mathcal{C}$ means taking complex conjugates, $|\underline{\omega}| = \sum_{i \in [k]} \omega_i$, and $\underline{\omega} \cdot \underline{h}$ is the dot product of $\underline{w}$ and $\underline{h}$.

We have the following proposition from \cite{Bienvenu-Shao-Teravainen}.
\begin{proposition}[\cite{Bienvenu-Shao-Teravainen}]
\label{thm:theta_decomposed}
There exists a constant $\delta > 0$ such that the following holds: For any $k \in \N, \epsilon > 0$ and a sufficiently large $N = N(\epsilon)$, letting $w = \log \log N, W = \prod_{p < w, p \in \P} p$, and $\theta_{W, -1}$ be defined as in \eqref{eq:w-trick-defined}, we can decompose
\[
    \theta_{W, -1} = f_1 + f_2 \text{ on } [N]
\]
such that $\delta \leq f_1 \leq 2$ pointwise and $\lVert f_2 \rVert_{U^{k+1}(\Z_N)} \leq \epsilon$. 
\end{proposition}
\begin{proof}
This proposition follows from the transference principle stated in \cite[Proposition 3.9]{Bienvenu-Shao-Teravainen}. Proposition 4.2 and Proposition 5.2 in the same paper show that for large $w$, the function $f = \theta_{W, -1}$ satisfies the conditions in the hypothesis of \cite[Proposition 3.9]{Bienvenu-Shao-Teravainen}, and so proves \cref{thm:theta_decomposed}. 
    % Since $\theta_{2, i}(n) = \theta_{2}(n - h_i)$, we have the corresponding $W$-tricked version is
    % \[
    %     \theta_{W, 1, i}(n) = \left( \frac{\phi(W)}{W} \right)^m \theta_2 (Wn - Wh_j + b)
    % \]
    % which is the shift of $\theta_{W, -1}$ by $Wh_j - b$. Thus, $\theta_{W, 1, i}$ can be decomposed into two functions $f_1, f_2$ which are the shifts of the corresponding decomposition of $\theta_{W, -1}$. Since shifting functions do not change their lower bounds and Gowers norms, the conclusion of the corollary holds for $\theta_{2, i}$. (We may need to use the fact that the shift amount $Wh_j - b$ is very small compared to $N$ to avoid the wrap around issue.)
\end{proof}

% \Pnote{No, $\theta_{W,1}$ certainly does not satisfy condition (ii) if $\theta=\theta_2$ and my paper does not say this. This would be true for $\theta_{W,b}$ where $b$ satisfies $(b+h_i,W)=1$ for all $h_i$. And $b=1$ does not. If $\theta=\theta_1$ neither, since $(b+2,W)=3\neq 1$.

% However for $\theta=\theta_1$ and $b=-1$, I can confirm that the statement is true. Thus actually $\P_{\mathrm{Chen}} +1$ is a set of recurrence.

% More generally, let $\P_{h,\ell}$ be the set of primes such that
% $p+h$ has at most $\ell$ prime factors.
% Then if $\ell >1+\Omega(h+1)$, then $\P_{h,\ell}-1$ is intersective.
% Indeed, let $h+1=p_1\ldots p_k$ be the factorisation into primes of $h+1$, with $\ell>1+k$. Then $\P_{h,\ell}-1\supset (\P-1)\cap (h+1)\cdot(\P\cdot\P-1)$. Now an element of $(\P-1)\cap (h+1)\cdot(\P\cdot\P-1)$
% is of the form $p-1=(h+1)(n-1)$ for $p$ prime, $n$ having at most two prime factors. 
% This condition involves two linear forms, $L_1(n)=1+(h+1)(n-1)$ and $L_2(n)=n$.
% This is an admissible pair of linear forms.
% Therefore our theorem applies just as when the forms were $n,n+2$.
% } 

\subsection{Intersectivity of shifted Chen primes}\

%\Anote{The next theorem is being fixed.}

The fact that $\P_{\mathrm{Chen}} + 1$ is intersective follows from a much more general result below, which says that $\P_{\mathrm{Chen}} + 1$ is a set of ``multiple polynomial recurrence for measure preserving systems of commuting transformations.'' In other words, by Furstenberg's Correspondence Principle \cite[Theorem 1.1]{Furstenberg77}, the intersectivity part of \cref{thm:Chen_primes_recurrence} corresponds to the case $\ell = m = 1$ and $q_{1, 1}(n) = n$ of the following proposition.

\begin{proposition}
\label{thm:main-twin-prime-recurrence}
Let $\ell \in \N$, $(X, \mathcal{B}, \mu)$ be a probability space and let $T_1, \ldots, T_{\ell}: X \to X$ be commuting measure preserving transformations (i.e. $\mu(T_i^{-1} B) = \mu(B)$ for all $i \in [\ell]$ and $B \in \mathcal{B}$). Let $m \in \N$ and $q_{i,j}: \Z \to \Z$ be polynomials with $q_{i,j}(0) = 0$ for $i \in [\ell]$ and $j \in [m]$. For any $A \in \mathcal{B}$ with $\mu(A) > 0$, the set 
\[
    \left\{n \in \N: \mu \left(A \cap \left(\prod_{i=1}^{\ell} T_i^{-q_{i,1}(n)} A \right) \cap \ldots \cap \left(\prod_{i=1}^{\ell} T_i^{-q_{i,m}(n)} A \right) \right) > 0 \right\}
\]
has a nonempty intersection with  $\P_{\mathrm{Chen}} + 1$.
\end{proposition}

\begin{proof}%[Proof of \cref{thm:main-twin-prime-recurrence}]
Define
\[
    \alpha(n) = \mu \left(A \cap \left( \prod_{i=1}^{\ell} T_i^{-q_{i,1(n)}} A \right) \cap \ldots \cap \left( \prod_{i=1}^{\ell} T_i^{-q_{i,m}(n)} A \right) \right).
\]
By \cite[Corollary 4.2]{Frantzikinakis_Host_Kra13} (see also \cite[Theorem 3.2]{Bergelson-Host-McCutcheon-Parreau}), there is a constant $c > 0$ depending only on $\mu(A), m$ and the polynomials $q_{i,j}$ such that for sufficiently large $N$, for any $W \in \N$,
\begin{equation}
\label{eq:uniform_bound_below}
    \E_{n \in [N]}  \alpha(Wn) \geq c.
\end{equation}
Let $\delta$ be the constant found in \cref{thm:theta_decomposed} and let $\epsilon > 0$ be very small compared to $c \delta$. By \cref{thm:theta_decomposed}, for sufficiently large $N$, we have
\begin{equation}
\label{eq:theta_f1f2}
    \E_{n \in [N]} \theta_{W,-1}(n) \alpha(Wn) = \E_{n \in [N]} f_1(n) \alpha(Wn) + \E_{n \in [N]} f_2(n) \alpha(Wn),
\end{equation}
where $f_1, f_2$ satisfy the conclusion of \cref{thm:theta_decomposed} with respect to aforementioned $\delta$ and $\epsilon$.  

Since $f_1$ is bounded below pointwise by $\delta$, \eqref{eq:uniform_bound_below} implies that
\begin{equation}
\label{eq:f_1}
    \E_{n \in [N]} f_1(n) \alpha(Wn) \geq c \delta.
\end{equation}
On the other hand, by \cite[Lemma 3.5]{Frantzikinakis_Host_Kra13}, there exists an integer $k$ that depends only on the maximum degree of the polynomials $q_{i,j}$ and the integers $\ell, m$ such that
\begin{equation}
\label{eq:f_2}
    \E_{n \in [N]} f_2(n) \alpha(Wn) \ll \lVert f_2 \cdot 1_{[N]} \rVert_{U^k(\Z_{k N})} \ll \lVert f_2 \rVert_{U^k(\Z_N)} \leq \epsilon.
\end{equation}
Combining \eqref{eq:theta_f1f2}, \eqref{eq:f_1} and \eqref{eq:f_2}, we deduce that for sufficiently large $N$, 
\[
    \E_{n \in [N]} \theta_{W,-1}(n) \alpha(Wn) \gg 1
\]
where the implicit constant depends only on $\mu(A), \ell, m$ and the $q_{i,j}$'s. 
Thus, there exists $n \in \N$ such that $\alpha(Wn) > 0$ and  $\theta_{W, -1}(n) > 0$.

Observe that $\theta_{W, -1}(n) >0$ implies $\theta(Wn - 1) > 0$, which in turn implies $Wn - 1 \in \P_{\mathrm{Chen}}$. Therefore, the fact that $\alpha(Wn) > 0$ and $\theta_{W, -1}(n) > 0$ implies $Wn \in \{m \in \N: \alpha(m) > 0\} \cap (\P_{\mathrm{Chen}} + 1)$. In particular, this intersection is nonempty and our theorem follows. 
\end{proof}

\subsection{Prime-intersectivity of shifted Chen primes}

\label{sec:multiple_recurrence_for_primes_shifted_chen}

Similarly to the case of $\N$-intersectivity, the fact that $\P_{\mathrm{Chen}} + 1$ is prime intersective (the second part of \cref{thm:Chen_primes_recurrence}) follows from a more general result concerning multiple recurrence.

\begin{proposition} \label{prop:Chen_shifted_for_primes}
Let $A \subset \P$ be such that $\overline{d}_{\P}(A) > 0$. Then for every $k \in \N$, there exists $p \in \P_{\mathrm{Chen}}$ such that 
\begin{equation*}\label{eq:recurrence_in_prime_gaps_twin}
    a, a + (p+1), \ldots, a + k(p+1) \in A.
\end{equation*} 
\end{proposition}
%\Anote{Need to verify $\P_{\mathrm{Chen}} -1$ or $\P_{\mathrm{Chen}} +1$ works.}
%\Anote{Above theorem is true for $\P_{\mathrm{Chen}}$ but is false for $\P_{\mathrm{bdd}}, h$ with $h \geq 3$ due to local obstruction.}

\begin{proof}
The idea is similar to the proof of \cref{thm:main-twin-prime-recurrence}. The main difference is that we will use the Green-Tao Theorem \cite{Green_Tao08} on arithmetic progressions in primes instead of the uniform Furstenberg-Katznelson Multiple Recurrence (\cite[Corollary 4.2]{Frantzikinakis_Host_Kra13}, \cite[Theorem 3.2]{Bergelson-Host-McCutcheon-Parreau}). %There are only some minor changes we need to make which reflect the fact that we are working with the set of primes, instead of the full set of natural numbers.

For the rest of the proof, $w$ will be a very large integer, $W = \prod_{p < w, p \in \P} p$, and $N \gg e^{e^w}$. 
These parameters can always be increased later without affecting our proof. 
%Let $w$ be a large integer and $W = \prod_{p < w, p \in \P} p$. 

Suppose $\overline{d}_{\P}(A) = \rho > 0$. For $b \in [W]$ coprime to $W$, define $f_{W, b} = \frac{\phi(W)}{W} \log (Wn + b) 1_{A}(Wn + b)$. 
Note that 
\begin{multline*}
    \overline{d}_{\P}(A) = \limsup_{N \to \infty} \frac{|A \cap \P \cap [N]|}{|\P \cap [N]|} = \limsup_{N \to \infty} \E_{n \in [N]} \log(n) 1_A(n) \\
    = \limsup_{N \to \infty} \E_{\substack{b \in [W] \\ (b, W) = 1}} \E_{n \in [N/W]} f_{W,b}(n).
\end{multline*}
Therefore, there exists $b \in [W]$ coprime to $W$ such that $\limsup_{N \to \infty} \E_{n \in [N]} f_{W,b}(n) \geq \rho$. 
%Fix these $W$ and $b$ and let $N$ be a very large integer compared to $W$. In the later stage, this parameter $N$ can always be increased if necessary. 
Define 
\[
    \alpha_{W, N}(n) = \E_{a \in [N]} f_{W, b}(a) f_{W, b}(a + n) \cdots f_{W, b}(a + kn). 
    %= \E_{a \in [N]} \prod_{i=0}^k f_{W,b}(a + in).
\]
By Theorem 3.5 and Proposition 9.1 in \cite{Green_Tao08} (or see page 524 in that paper), there exists a positive constant $c = c(k, \rho)$ such that
\begin{equation}\label{eq:alpha_bounded_below_prime}
    \E_{n \in [N]} \alpha_{W, N}(n) \geq c.
\end{equation}
Define $\theta_{W, -1}$ as in \eqref{eq:w-trick-defined}. Let $\delta$ be the constant found in \cref{thm:theta_decomposed} and let $\epsilon$ be very small compared to $c \delta$. In view of \cref{thm:theta_decomposed} and by increasing $w$ and $N$ if necessary, $\theta_{W, -1}$ can be decomposed on $[N]$ as
\[
    \theta_{W, -1} = f_1 + f_2
\]
where $\delta \leq f_1 \leq 2$ pointwise and $\lVert f_2 \rVert_{U^{k+1}(\Z_N)} \leq \epsilon$.
It follows from \eqref{eq:alpha_bounded_below_prime} that
\begin{equation}\label{f_1alpha}
    \E_{n \in [N]} f_1(n) \alpha_{W, N}(n) \geq c \delta.
\end{equation}
In the equation above, we use the fact that $\delta, c$ only depend on $\rho$ and $k$. As a result, $c \delta$ does not depend on $w$ or $N$.

Now we want to show
\begin{equation}\label{f_2alpha}
    \E_{n \in [N]} f_2(n) \alpha_{W, N} (n) = \E_{n \in [N]} f_2(n) \E_{a \in [N]} \prod_{j=0}^k f_{W,b}(a + jn) \ll \epsilon.
\end{equation}
In order to do this, we will use the generalized von Neumann Theorem (van der Corput Lemma) (see \cite[Proposition 5.3]{Green_Tao08}, \cite[Proposition 7.1]{Green_Tao12} or \cite[Proposition 3.8]{Bienvenu-Shao-Teravainen}) which says that if the integers $a_j, b_j$ satisfy $a_i b_j \neq a_j b_i$ for all $i \neq j$ and if  $|f_j| \leq \nu + 2$ for all $j$ where $\nu$ is a ``pseudorandom measure'' on $[N]$, then
\[
    \E_{m, n \in [N]} \prod_{j=0}^k f_j(a_j m + b_j n) \ll \inf_{0 \leq j \leq k} \lVert f_j \rVert_{U^{k}(\Z_N)}.
\]
%(Note that theorem is only stated for $a_j = 1$ and $b_j = j$, but the proof is the same for the general case.)
By \cite[Theorem 9.1]{Green_Tao08}, $|f_{W, b}|$ is bounded by a pseudorandom measure $\nu$. It remains to check that $|f_2| \leq \nu + 2$. Since $f_2 = \theta_{W, -1} - f_1$ and $0 \leq f_1 \leq 2$, it suffices to check that $0 \leq \theta_{W, -1} \leq \nu$. This in turn was shown in \cite[Proposition 4.2]{Bienvenu-Shao-Teravainen}. 

%\Anote{Need to check that the sum (or rather the average) of the pseudorandom measures in \cite{Green_Tao08} and in \cite{Bienvenu-Shao-Teravainen} is still a pseudorandom measure.}

Putting \eqref{f_1alpha} and \eqref{f_2alpha} together, we get 
\[
    \E_{n \in [N]} \theta_{W,-1}(n) \alpha_{W, N}(n) \gg 1.
\]
In particular, there exist $a, n \in [N]$ such that
\[
    Wa - 1, W(a + n) - 1, \ldots, W(a + kn) - 1 \in A
\]
and $Wn \in \P_{\mathrm{Chen}} + 1$. Thus $Wa - 1, Wa - 1 + Wn, \ldots, Wa - 1 + k(Wn) \in A$ where $Wn \in \P_{\mathrm{Chen}} + 1$. 

% Similar to the proof of \cref{thm:main-twin-prime-recurrence}, we have $\P_{\mathcal{H}} = \bigcup_{1 \leq i < j \leq m} \P_{\mathcal{H}, i, j}$ where
% \[
%     \P_{\mathcal{H}, i, j} = \{n \in \N: n + h_i, n + h_j \in \P\}.
% \]
% By \cref{lem:partition_regular_primes}, there exist $1 \leq i < j \leq m$ such that $\P_{\mathcal{H}, i, j} - 1$ is a set of multiple recurrence with respect to primes. Since $\P_{\mathcal{H}, i, j} - 1 \subset \P - h_i - 1$, it must be that $h_i = 0$. (Otherwise, $\P - h_i - 1$ will not contain infinitely many multiple of $h_i + 1$ and we can choose the $A \subset P$ to contain all primes that are congruent $1 \bmod h_1 + 1$. In this situation, every element of $A - A$ is a multiple of $h_i + 1$.) Thus,
% \[
%     \P_{\mathcal{H}, i, j} = \{n \in \N: n, n + h_j \in \P\} = \P_{\mathrm{bdd}, h_j}.
% \]
% In other words, $\P_{\mathrm{bdd}, h_j} - 1$ is a set of multiple recurrence with respect to primes. 
\end{proof}

% \begin{remark}
% Above proof works equally well if we replace $A \subset \P$ by $A \subset \P_{\mathrm{Chen}}$. In this case, we need to use \cite{Bienvenu-Shao-Teravainen} instead of \cite{Green_Tao08}. As a result, we can show $\P_{\mathrm{Chen}} + 1$ is $\P_{\mathrm{Chen}}$-intersective, and more generally, is a set of multiple recurrence with respect to $\P_{\mathrm{Chen}}$. 
% \end{remark}

\begin{remark}
    In this remark, we explain why the proofs of Propositions \ref{thm:main-twin-prime-recurrence} and \ref{prop:Chen_shifted_for_primes} do not work for $\P_{\mathrm{Chen}} - 1$. A main ingredient in these proofs is \cref{thm:theta_decomposed}, which in turn uses \cite[Proposition 5.2]{Bienvenu-Shao-Teravainen}. The hypothesis of \cite[Proposition 5.2]{Bienvenu-Shao-Teravainen} requires both $b$ and $b + 2$ to be coprime to $W$, where $W = \prod_{p \in \P, p < w} p$ comes from the $W$-tricked weighted indicator function of Chen primes $\theta_{W, b}$. (This requirement  is due ultimately to the fact that $\theta_{W, b}$ is supported on $\{n: Wn + b \text{ is a Chen prime}\}$.) If we work with $\P_{\mathrm{Chen}} + 1$, then $b = -1$ and $b + 2 = 1$, both of which are coprime to $W$. On the other hand, if we work with $\P_{\mathrm{Chen}} - 1$, then $b = 1$ and so $b + 2 = 3$ which is not coprime to $W$. At the moment, we do not know how to remove this local obstruction.
\end{remark}

	\section{A quantitative approach and Proof of \texorpdfstring{\cref{th:prime-elem}}{Theorem B}}

 \label{sec:quantitative-approach}

\subsection{Proof of \texorpdfstring{\cref{th:prime-elem}}{Theorem B}}
For two functions $f, g: \R_{>0} \to \R_{>0}$, the expression $f(x) = o(g(x))$ means $\lim_{x \to \infty} f(x)/g(x) = 0$.
Theorem \ref{th:prime-elem} will be a simple consequence of the following proposition.

\begin{proposition}
\label{lm:generalisation}
Let $E \subset \N$ be an infinite set and $E(x)=|E\cap [1,x]|$ its counting function. Let $E_m(x)=|\{n\in E \cap [1, x]: n+m\in E\}|$. Assume that 
there is a thick set $T \subset \N$ and an increasing bijection $f :\R_{>0}\rightarrow\R_{>0}$ such that the
following hold as $x$ tends to infinity:
%\Pnote{it suffices if it works as $x$ tends to infinity along a subsequence}
\begin{enumerate}
    \item \label{hyp1}    $f^{-1}(x)=o(E(x))$.
    \item \label{hyp2} $\max \{E_m(x):m\in T,m\leq f(x)\}=o(E(x))$.
\end{enumerate}
Then there exists $A\subset E$ of relative density 1 such that
$A-A$ is not syndetic.

If both conditions are satisfied as $x$ tends to infinity along a common subsequence,
then there exists $A\subset E$ of upper relative density 1 such that
$A-A$ is not syndetic.

%If additionally $T=\N$ and $G(x)=x^{o(1)}$, if furthermore for any pair $(a,b)\in \mathcal{S}^2$ and $m\in \N$ we have $a-b=m\Rightarrow a,b\ll m^{O(1)}$, then there
%is a thick set which is not chromatically $\mathcal{S}$-intersective.
\end{proposition}
\begin{remark}
The hypothesis of \cref{lm:generalisation} may look unnatural and complicated, but certainly some sort of conditions on $E_m(x)$ is necessary.
We discuss this issue in \cref{sec:concluding_remarks}.
\end{remark}

\begin{proof}[Proof of \cref{lm:generalisation}]
%Let us proof the first part of the theorem.
 By definition, the thick set $T$ contains a set of the form $R=\bigcup_{k\in\N} I_k$ where $I_k=\N\cap [g(k)-k,g(k)]$ and $g:\R_{>0}\rightarrow\R_{>0}$ is an increasing bijection. We
    may choose $g$ to grow as fast as we like, and we will determine 
    a suitable rate of growth later.

We shall construct a set $C\subset E$ of relative density 1 in $E$ such that
$B=(C+R)\cap E$ is relatively sparse in $E$.
Then $A=C\setminus B$ satisfies $A\cap (A+R)\subset (C+R)\cap (E\setminus B) =\varnothing$, so
$(A-A)\cap R=\varnothing$, in particular $A-A$ is not syndetic.

Observe that for any $C\subset E$ we have
$$
C+R=\bigcup_{\substack{s\in C,k\in\N\\g(k)\leq f(s)}}(s+I_k)\cup\bigcup_{\substack{s\in C,k\in\N\\g(k)> f(s)}}(s+I_k)
$$
Now we try to construct $C$ such that 
$\bigcup_{\substack{s\in C,k \in \N\\g(k)\leq f(s)}}(s+I_k)$ does not meet $E$.

    Define
    \[
        C = \{s \in E: (s + I_k) \cap E = \varnothing \text{ for all } k \text{ such that } g(k) \leq f(s) \}.
    \]
Let $E'(x) := |(E \setminus C) \cap [x]|$ be the counting function of $E \setminus C$. If $s \in E \setminus C$, then by definition of $C$, there exists $k \in \N$ such that $g(k) \leq f(s)$ (so $k \leq g^{-1}(f(s))\,$) and $m \in I_k \subset T$ such that $s + m \in E$.  Therefore,
    \[
    E'(x)\leq \sum_{k:g(k)\leq f(x)}\sum_{m\in I_k}E_m(x) < g^{-1}(f(x))^2 \cdot \max\{E_m(x): m < f(x),m\in T\} = o(E(x))
    \]
    if $g$ grows sufficiently quickly, using hypothesis (\ref{hyp2}).

    Then
$(C+R)\cap E\subset \bigcup_{\substack{s\in E,k \in \N\\g(k)> f(s)}}(s+I_k)=:D$.
We show that this set is sparse.
Observe that if $g(k) > f(s)$, then $s < f^{-1}(g(k))$ and $s + I_k \subset [g(k)-k, g(k) + f^{-1}(g(k))]$. Therefore,
$$
D\subset \bigcup_{k}[g(k)-k,g(k)+f^{-1}(g(k))].
$$
In particular,
$D\cap [1,x]\subset \bigcup_{g(k)<x}[g(k)-g^{-1}(x),g(k)+f^{-1}(x)]$, whose cardinality is 
at most 
$$g^{-1}(x)  \cdot ((f^{-1}(x)+g^{-1}(x)) = o(E(x))$$
if $g$ grows sufficiently quickly (in terms of $f$ and $E$), in view of hypothesis (1).

It follows that the counting function $B(x)$ of $B = (C+R)\cap E$ satisfies $B(x)=o(E(x))$, as desired, and we are done.

If the hypotheses only hold for $x$ in an increasing sequence $(x_n)_{n\in\N}$ of integers, we still have $E'(x_n)=o(E(x_n))$ and also
$B(x_n)=o(E(x_n))$ as $n$ tends to infinity, which concludes the proof.
\end{proof}	

\begin{proof}[Proof of \cref{th:prime-elem}] 
To prove Theorem \ref{th:prime-elem}, it suffices
to check the hypotheses of \cref{lm:generalisation} for the set of primes.
In this case, one may take $T=\N$ and $f:x\mapsto x^2$. The counting function
of the prime $E(x)$ is asymptotic to $x/\log x$ by the prime number theorem,
so the first hypothesis is satisfied.

		By definition, $E_m(x)$ is the number of primes $p\leq x$ such that
$p+m$ is prime.		
		We know, by Selberg's sieve, that $$E_m(x)\ll \frac{x}{\log^2x}\prod_{p\mid m}(1+1/p),$$ where the implied constant is absolute; see for instance \cite{HR}.
		Furthermore, denoting the number of prime factors of $m$ by $\omega(m)$
		and observing that $\omega(m)\leq \log_2 m$
		and $\prod_{p\leq n}(1+1/p)\ll \log n$,
		we have
		$$
		\prod_{p\mid m}(1+1/p)\leq \prod_{p\leq  \omega(m)}(1+1/p)
		\ll \log\log m.
		$$
Thus for any $m\leq x^2$, we have $E_m(x)\ll \frac{x\log\log x}{\log^2x}=o(E(x))$, confirming the second hypothesis.
\end{proof}

\cref{lm:generalisation} also applies to $E = \{P(n): n \in \N\}$ where $P \in \Z[x]$ with a positive leading coefficient and $d = \deg(P) \geq 2$. In this case, we take $T = \N$ and $f: x \to x^{d + 1}$. Since $P(x) = o(x^{d+1})$, we have $f^{-1}(x) = x^{1/(d+1)} = o(E(x))$. For $m \in \N$, $E_m(x) \ll d(m)$ where $d(m)$ is the number of divisors of $m$. Applying the divisor the bound $d(m)=m^{o(1)}$, it is easy to see the second hypothesis is true in this case. For the same reason, we can show that \cref{lm:generalisation} also applies to $E = \{P(p): p \in \P\}$ where $P \in \Z[x]$ having a positive leading coefficient and degree $\geq 2$.

\subsection{A discussion on \texorpdfstring{\cref{lm:generalisation}}{Proposition 3.1}}

\label{sec:concluding_remarks}

%\Anote{Change $\cS$ to $E$.}
 
%\subsection{The range of application of Lemma \ref{lm:generalisation} and Theorem \ref{thm:main-good-recurrence}}\label{hypotheses}

We construct in \cref{lm:generalisation} a set $A \subset E$ such that $d_E(A) = 1$ and $A-A$ is not syndetic, assuming some more or less complicated hypotheses on $E$.
%It applies for instance to the primes and to many other interesting arithmetic sets. 
Here we discuss the necessity of certain hypotheses of \cref{lm:generalisation}.

As mentioned in the introduction, if $\underline{d}(E) > 0$ and $A \subset E$ satisfies $\overline{d}_E(A) > 0$, then $A - A$ is syndetic. Therefore, it is necessary in \cref{lm:generalisation} that $\underline{d}(E) = 0$. However, merely having zero lower density does not guarantee the conclusion of \cref{lm:generalisation} as we shall see; in fact, even having upper Banach density 0 does not suffice.

First, recall the notations $E(x)$ and $E_m(x)$ from \cref{lm:generalisation}. If for some $m\in\N$ and constant $c_m>0$ and every $x>0$ we have $E_m(x) \geq c_m E(x)$, then every $A\subset E$ such that $m \not\in A-A$ must have $\overline{d}_E(A)<1-c_m/2$. 
Indeed, for every pair $\{n,n+m\}\subset E$, at least one element of the pair must be outside of $A$; and any element of $E$ may be in at most two such pairs. 
This observation already makes the condition on $E_m(x)$ in \cref{lm:generalisation} less surprising and will be implicitly at the core of the next proposition.

\begin{proposition} \label{prop:ss'-z}
Let $E$ be the set of all positive integers which have the same number of $0$'s and $1$'s in their binary expansions; that is 
\[
E = \bigcup_{k=0}^\infty \left\{ 2^{2k+1} + 
\sum_{i \in I} 2^i: I \subset [0,2k], |I|=k \right\}.
\]
Then $E$ has the following properties:
\begin{enumerate}
\item  \label{item:Banach_density_0_prop7.1}  $d^*(E)  = 0$,
\item \label{item:E'-E'=Z} If $E' \subset E$ is such that $\overline{d}_{E}(E')=1$, then $E'-E' = \Z$.
\end{enumerate}
\end{proposition}
\begin{proof}   
First we prove (\ref{item:Banach_density_0_prop7.1}). Note that
\begin{equation}\label{eq:E_2^n_bound}
|E \cap [1, 2^{2n+2}]| = \sum_{k =0}^n \binom{2k}{k} = O(2^{2n} /\sqrt{n}),
\end{equation}
so $E$ has upper density $0$.

We will now prove the stronger statement that $E$ has upper Banach density $0$. It suffices to show that as $n \rightarrow \infty$, for every $u \in \N$, we have $|E \cap [u, u+2^n)| = o(2^n)$. Also, we may assume that $u$ is divisible by $2^n$. We consider two cases.

\noindent \textbf{Case 1}: $u = 0$. In this case the claim follows from the estimate \eqref{eq:E_2^n_bound}.

\noindent \textbf{Case 2}: $u \geq 2^n$. We write
\begin{equation*} 
    u = \sum_{s=1}^\ell 2^{j_s}
\end{equation*}
where $n \leq j_1 < \cdots < j_\ell$. 
Suppose $v \in E \cap [u, u+2^n)$. By the definition of $E$,
\begin{equation*} 
    v = 2^{2k+1} + 
\sum_{r=1}^k 2^{i_r}
\end{equation*}
for some $0 \leq i_1 < \cdots  < i_k \leq 2k$. Therefore
\begin{equation*} 
    v = 2^{2k+1} + 
\sum_{r=1}^k 2^{i_r} = \sum_{s=1}^\ell 2^{j_s} + y
\end{equation*}
for some $0 \leq y < 2^{n}$. 
We must necessarily have $2k+1 = j_\ell$ (in particular, $k$ is uniquely determined in terms of $u$). Furthermore, $i_k = j_{\ell-1}, \ldots, i_{k-\ell + 2} = j_1$ and they are also uniquely determined.
Hence, 
\[
y = \sum_{r=1}^{k-\ell + 1} 2^{i_r}.
\]
Therefore, 
\[
    |E \cap [u, u+2^n)| \leq \binom{n}{k - \ell + 1} = o(2^n).
\]

We now proceed to prove (\ref{item:E'-E'=Z}). Let $a \in \N$ be arbitrary. Suppose $a = \sum_{i \in A} 2^i$ is the binary representation of $a$. Then we have
\[
a = \sum_{i \in A + 1} 2^i - \sum_{i \in A } 2^i.
\]

Let $|A| =m$ and $\ell = \max A$. By \eqref{eq:E_2^n_bound}, there exists a constant $c_{\ell, m} >0$ such that for all $n$ sufficiently large, we have \[
 \binom{2n - \ell-1}{n-m} > c_{\ell, m} |E \cap [1, 2^{2n+4}]|.
\]

Since $\overline{d}_{E}(E') = 1$, we have $|E' \cap [1, x]| \geq (1 - \frac{c_{\ell, m}}{2}) |E \cap [1, x]|$ for infinitely many $x \in \N$. Let $x$ be any such number and let $n$ be such that $2^{2n+2} < x \leq 2^{2n+4}$.

Note that for any subset $J \subset [\ell+2, 2n]$ with $|J| = n-m$, we have
\[
a = s_1 - s_2 = \left( 2^{2n+1} + \sum_{i \in J \cup (A+1)} 2^i \right) - \left( 2^{2n+1} + \sum_{i \in J \cup A} 2^i \right)
\]
is a difference of two elements $s_1, s_2 \in E \cap [1,2^{2n+2}] \subset E \cap [1,x]$. The number of such representations is 
\[
 \binom{2n - \ell-1}{n-m} > c_{\ell, m} \left|E \cap \left[1, 2^{2n+4} \right] \right| \geq c_{\ell, m} \left|E \cap [1,x] \right|
\]
if $n$ is sufficiently large. 

Since $|E' \cap [1, x]| \geq (1 - \frac{c_{\ell, m}}{2}) |E \cap [1, x]|$, the number of representations $a = s_1 - s_2$ with $s_1, s_2 \in E \cap [1,x]$ and at least one of them not belonging in $E' \cap [1, x]$, is at most  $c_{\ell, m} |E \cap [1, x]|$. Thus there exists a representation $a = s_1 - s_2$ where $s_1, s_2 \in E'\cap [1,x]$. This shows that $E'-E' = \Z$, and we are done.
\end{proof}

With the notation from \cref{lm:generalisation}, in \cref{prop:ss'-z}, we just proved $E_a(x) \geq c_a E(x)$ for some $c_a>0$, for any $a$, and for
$x$ large enough in terms of $a$.
Yet for many $a$, the constant $c_a$ may be expected to be very small,
in view of Cusick's conjecture and partial results towards it such as
\cite{digits}.
When $c_a$ is positive but indeed very small on a sufficiently large set of integers $a$, we have the following relaxation of \cref{lm:generalisation} (in which both hypothesis and conclusion are slightly weaker compared to \cref{lm:generalisation}). 

\begin{proposition}
\label{lm:generalisationEps}
Let $E$ be a set of positive integers and $E(x)=|E\cap [1,x]|$ its counting function. Let $E_m(x)=|\{n\in E \cap [1, x]: n+m\in E\}|$. Assume that 
there is a sequence $(c_m)_{m\in \N}$ and an increasing bijection $f :\R_{>0}\rightarrow\R_{>0}$ such that the
following hold:
%\Pnote{again we may add ''along a subsequence''}
\begin{enumerate}
    \item $f^{-1}(x)=o(E(x))$ as $x$ tends to infinity.
    \item \label{hypii} $\forall m\in\N,\,\forall
    x>f^{-1}(m),$ one has $E_m(x)\leq c_m\cdot E(x)$.
\end{enumerate}
Assume additionally that 
there is a thick set $T$ such that $\lim_{m\in T,m\rightarrow\infty}c_m=0$.
Then for every $\epsilon >0$ there exists $A \subset E$ of lower relative density at least $1-\epsilon$ such that
$A-A$ is not syndetic. 
\end{proposition}
We omit the proof since it is extremely close to that of \cref{lm:generalisation}. 
Note that given a sequence $(c_m)$ of positive real numbers, the negation of the statement
that there exists a thick set $T$ such that $\lim_{m\in T,m\rightarrow\infty}c_m=0$
is the statement
that there exists $\eta>0$ such that $\{m\in\N:c_m>\eta\}$ is syndetic.

In practice, \cref{lm:generalisationEps} is unwieldy because the hypothesis \eqref{hypii}
is very difficult to prove.
A more manageable hypothesis would be for instance $\lim_{x\rightarrow\infty} E_m(x)/E(x)\leq c_m$, but it is not clear whether such an hypothesis is sufficient.
On the other hand, the negation of an hypothesis of this form leads to a conclusion
of a positive upper Banach density, as stated in the proposition below, of which we also omit the proof.
\begin{proposition}
\label{th:largecm}
Let $E\subset\N$.
Suppose $c_m:=\liminf_{x \to \infty} E_m(x)/E(x)$ is such that
$\{m:c_m>\eta\}$ is syndetic (or even has positive lower density) for some $\eta>0$.
Then $d^*(E) > 0$.
\end{proposition}

\section{Bohr compactification and proof of Theorem \ref{thm:pws}}
\label{sec:proof_of_bohr_compactification}
	
Let $\T = \R/\Z$ be the $1$-dimensional torus and $\T_d$ be $\T$ endowed with the discrete topology. The \emph{Bohr compactification} of $\Z$ is the Pontryagin dual of $\T_d$ and denoted by $b\Z$. Then $b \Z$ is a compact abelian group. For every $n \in \Z$, let $\tau(n) \in b\Z$ be the character on $\T_d$ defined by $\tau(n)(\chi) = \chi(n)$, for every $\chi \in \T \cong \hat{\Z}$. Then $\tau(\Z)$ is dense in $b\Z$ (for a proof, see \cite[Theorem 1.8.2]{rudin}), and as a result $\tau(\N)$ is also dense in $b \Z$. 
%\Pnote{What is $b\N$?}\Anote{Fixed. It should be $b\Z$.}
%\footnote{For any compact group $K$, if $\tau: \Z \to K$ is a homomorphism such that $\tau(\Z)$ is dense in $K$, then $\tau(\N)$ is dense. Indeed, first notice that $K$ must be abelian. Let $m \in \Z$ be arbitrary and $U \subset K$ be a neighborhood of $\tau(m)$. Then $V = U - \tau(m)$ is a neighborhood of $0_K$ and as a result, $V \cap (-V)$ is a neighborhood of $0_K$. Since $V \cap (-V)$ is symmetric, if $\tau(n) \in V \cap (-V)$, $\tau(-n) \in V \cap (-V)$. It follows that there are infinitely many $n \in \N$ for which $\tau(n) \in V \cap (-V)$. In particular, there is $n > -m$ such that $\tau(n) \in V \cap (-V) \subset V$. Then $n + m \in \N$ and $\tau(n+m) = \tau(n) + \tau(m) \in V + \tau(m) \in U$.} 
We use $m_{b\Z}$ to 
denote the normalized Haar measure on $b\Z$. We remark that $(b\Z, \tau)$ has the following universal property: if $K$ is any compact Hausdorff topological group and $\phi: \Z \to K$ is a homomorphism, then there is a unique continuous homomorphism $\tilde{\phi}: b\Z \to K$ such that $\phi = \tilde{\phi} \circ \tau$. 

%and every Bohr almost periodic function $f$ can be written as $f = g \circ \tau$, where $g$ is a continuous function on $b\Z$. 

%The homomorphism $\tau$ is universal with respect to homomorphisms into compact Hausdorff groups; that is if $K$ is another compact Hausdorff group and $\pi: \Z \to K$ is a homomorphism, then there is a unique continuous homomorphism $\tilde{\pi}: b\Z \to K$ such that $\pi = \tilde{\pi} \circ \tau$. 
%The Bohr compactification also has a concrete description; it is the dual of $\widehat{\Z}$ where $\widehat{G}$ is given the discrete topology (see \cite{gll}). 

%See \cite{rudin} for basic results on the Bohr compactification and \cite{bjorklund-fish} for a recent application to sumsets. 

We say a sequence $E = \{c_n : n\in \N\} \subset \N$ with $c_1 < c_2 < \cdots$ is good for the pointwise ergodic theorem  if for any measure preserving system $(X, \mathcal{B}, \mu, T)$, for any $f \in L^\infty(X)$, the pointwise limit $\lim_{N \rightarrow \infty} \frac{1}{N} \sum_{n=1}^N f(T^{c_n} x)$ exists for almost all $x \in X$. 
  
By works of Bourgain \cite{bourgain}, Wierdl \cite{wierdl} and Nair \cite{nair}, the following sequences are good for the pointwise ergodic theorem: $\{ P(n): n \in \N \}, \{ \lfloor Q(n) \rfloor: n \in \N \}, \{p_n: n \in \N\}, \{P(p_n): n \in \N\}$. Here $P$ is any polynomial in $\Z[x]$, $Q$ is any polynomial in $\R[x]$, $p_n$ is the $n$-th prime and $\lfloor \cdot \rfloor$ denotes the integer part.

\cref{thm:pws} follows from a more general result. 
\begin{proposition}
\label{thm:main-good-recurrence}
Let $E \subset \N$ be such that its closure in $b\Z$ has measure $0$, i.e. $m_{b\Z}(\overline{\tau(E)}) = 0$.
Then for every $\epsilon > 0$, there exist $A \subset E$ with $\overline{d}_{E}(A) > 1 - \epsilon$ such that $E - A$ is not piecewise syndetic. In particular, $R = \N \setminus (E - A)$ is thickly syndetic but is not $E$-intersective.

Furthermore, if the natural enumeration of $E$ is good for the pointwise ergodic theorem, then $\overline{d}_{E}(A)$ can be replaced by $d_{E}(A)$.
\end{proposition}

	In \cite{Dressler-Pigno-Haarmeasure}, Dressler and Pigno showed that the closures of the following sets in $b\Z$ have measure zero:
\begin{enumerate}
    \item The set of prime powers $\{p^n: p \in \P, n \in \N\}$. This explains why \cref{thm:pws} follows from \cref{thm:main-good-recurrence}.
    
    \item The set of sums of two squares $\{x^2 + y^2: x, y \in \Z\}$.
    
    \item The set of square-full numbers, that is, the set of numbers $n$ so that every exponent in the prime factorization of $n$ is at least two.
\end{enumerate}
In Appendix \ref{sec:closure_measure_zero}, we give more examples of such sets:

\begin{enumerate}[resume]
    \item The set of values of a polynomial of degree >1, i.e. $\{ P(n) : n \in \Z \}$, where $P \in \Z[x], \deg P > 1$.
    
    \item The set of values of a binary quadratic form, i.e. $\{ a x^2 + b xy + c y^2: x, y \in \Z\}$, whose discriminant $D = b^2 - 4 ac$ is not a perfect square.
    
    \item More generally, the set of integers represented by a norm form,\footnote{A norm form is a homogeneous form $F(x_1, \ldots, x_d)= N_{K/\Q}(x_1 \omega_1 + \cdots + x_d \omega_d)$, where $K$ is an algebraic number field of degree $d \geq 2, \{\omega_1 , \ldots, \omega_d\}$ is a basis of the ring of integers of $K$ as a $\Z$-module, and $N_{K/\Q}$ denotes the norm.
    } e.g. $\{x^3 + 2y^3 + 4z^3 - 6xyz: x, y, z \in \Z \}$.
\end{enumerate}

The fact that all the examples of set $E$ presented above have zero Banach density is not a coincidence. It is because for $E \subset \N$, we always have $m_{b\Z}(\overline{\tau(E)}) \geq d^*(E)$, and so if $E$ has positive upper Banach density, it will not satisfy the hypothesis of \cref{thm:main-good-recurrence}.

As mentioned above, Theorem \ref{thm:main-good-recurrence} applies to $E=\{P(n):n\in\N\}$, where $P \in \Z[x]$ has degree $\geq 2$. Note that for such sets, we also have Theorem \ref{th:squares}, which says that there exists a thick set
(however, not thickly syndetic) which is not chromatically $E$-intersective.

	In the proof
of Theorem \ref{thm:main-good-recurrence},	
	we make use of the following two lemmas from \cite{griesmer-dense-set}. For completeness we include the short proofs. Note that the compact abelian groups appearing in this section are not assumed to be metrizable.

 The first lemma says that we can create sumsets with large measure and empty interior, if the group is separable.
 
	\begin{lemma}[{\cite[Lemma 2.7]{griesmer-dense-set}}]
		\label{lem:empty-interior}
		Let $K$ be a separable compact abelian group with the normalized Haar measure $m_K$ and let $E \subset K$ be a compact set with $m_K(E) = 0$. For all $\epsilon > 0$, there exists a compact set $F \subset K$ with $m_K(F) > 1 - \epsilon$ such that $E + F$ has empty interior. 
	\end{lemma}
	 
\begin{proof}
    Let $X=\{ x_n \}_{n=1}^\infty$ be a dense, countable subset of $K$. 
    By taking $G=K \setminus \bigcup_{n=1}^\infty (x_n - E)$, we have $x_n \not \in E + G$ for all $n$ and since $m_K(E) = 0$, $m_K(G)=1$.
    By regularity of the Haar measure on a compact group, $G$ contains a compact set $F$ of measure more than $1-\epsilon$,
    and $(E + F)\cap X\subset (E + G) \cap X = \varnothing$.
    Thus $E + F$ has empty interior.
%\Pnote{This seems to be correct, generic, simple, does not even use that $E$ is compact.}
\end{proof}

The next lemma says that the set of return times to a closed set with empty interior, is not piecewise syndetic.
	
\begin{lemma}{\cite[Lemma 4.1]{griesmer-dense-set}}
	\label{lem:pw-syndetic}
		Let $K$ be a compact abelian group and $\tau: \Z \to K$ be a homomorphism such that $\tau(\N)$ is dense in $K$. Let $U$ be a compact subset of $K$ with empty interior. Then the set  $\{n \in \N: \tau(n) \in U\}$ is not piecewise syndetic.
\end{lemma}

 \begin{proof} Suppose for a contradiction that $R:=\{n \in \N: \tau(n) \in U\}$ is piecewise syndetic. Then there is a finite set $A \subset \N$ such that $R': = \bigcup_{a \in A} (R+a)$ is thick. 
 %\Hnote{I can see this in $\Z$ in terms of the gaps, but is it true in an arbitrary group?} \Anote{A piecewise syndetic set is the intersection of a thick set and a syndetic set. This definition applies to arbitrary (abelian) groups. Thus I think the previous statement is true for arbitrary groups.} 
 Note that $R' = \{n \in \N: \tau(n) \in \bigcup_{a \in A} (U + \tau(a) )\}$. 
 
 We claim that $K = \bigcup_{a \in A} (U + \tau(a) )$. Suppose this is not true. Then $V: = K \setminus \bigcup_{a \in A} (U + \tau(a) )$ is nonempty and open (since $U$ is compact). Since $\tau(\N)$ is dense in $K$, the action of $\N$ on $K$ given by $T_n(x) = x + \tau(n)$ for all $x \in K$, defines a minimal topological dynamical system. By Birkhoff's theorem (see \cite[Theorem 1.15]{Furstenberg81}), the set $\N \setminus R' = \{n \in \N: \tau(n) \in V\}$ is syndetic. But this contradicts the fact that $R'$ is thick.
 
 Hence, $K = \bigcup_{a \in A} (U + \tau(a) )$, so one of the $U + \tau(a)$ has a non-empty interior, and $U$ has a non-empty interior. This is a contradiction.
 \end{proof}
 
In order to prove  \cref{thm:main-good-recurrence}, we need a new lemma.	\begin{lemma}\label{lem:SameIntegral}
		Let $K$ be a compact abelian group and $\tau: \Z \to K$ an arbitrary map.
  %be a homomorphism such that $\tau(\Z)$ is dense in $K$. 
  Then for every $E \subset \N$ and every measurable set $D \subset K$, there exists $z \in K$ such that 
		\[
		\overline{d}_{E}(\{n \in \N: z - \tau(n) \in D\}) \geq m_K(D).
		\]
		
		% If $\mb c\subset \N$ is good for pointwise convergence for group rotations, $(X,\mu,T)$ is a measure preserving $\Z$-system, and $f\in L^\infty(\mu)$, then for all $\varepsilon>0$, the set $E_\varepsilon:=\{x\in X: f_{\mb c}(x)>\int_X f\, d \mu - \varepsilon\}$ has positive measure.  In particular, if $D \subset X$, there is an $x\in X$ such that
		% \[
		% d_{\mb c}(\{n\in \Z: T^n x \in D\})> \mu(D)-\varepsilon.
		% \]
	\end{lemma}
	
	\begin{proof}
		Enumerate $E$ as the sequence $1 \leq c_1 < c_2 < \ldots$. Let $f = 1_D$, the indicator function of $D$, and for $N \in \N$, define the function $f_N: K \to [0, 1]$ by
		\[
		  f_N(z) := \frac{1}{N} \sum_{n = 1}^N f(z - \tau(c_n)).
		\]
		Since $m_K$ is translation invariant, $\int_K f_N \ d m_K = \int_K f \ d m_K$. Because $f_N$ is bounded, Fatou's lemma implies
		\[
		\int_K \limsup_{N \to \infty} f_N \ d m_K \geq \limsup_{N \to \infty} \int_K f_N \ d m_K = \int_K f \ d m_K = m_K(D).
		\]
		Therefore, the set 
		\[
		S:= \left \{z \in K: \limsup_{N \to \infty} f_N \geq m_K(D) \right\} = \left \{z \in K: \limsup_{N \to \infty} \frac{1}{N} \sum_{n = 1}^N 1_D(z - \tau(c_n)) \geq m_K(D) \right\}
		\]
		has positive measure; in particular, this set is non-empty.
		
		% Let $D \subset X$ be a measurable set and define $f = 1_D$. 
		Let $z$ be a point in $S$ and $A: = \{n \in \N: z - \tau(n) \in D\}$. Then 
		\[
		|A \cap \{c_1, \ldots, c_N\}| = \sum_{n = 1}^N 1_D(z - \tau(c_n)).
		\]
		Therefore, 
		\[
		\overline{d}_{E}(A) := \limsup_{N \to \infty} \frac{|A \cap \{c_1, \ldots, c_N\}|}{N} = \limsup_{N \to \infty} \frac{1}{N} \sum_{n = 1}^N 1_D(z - \tau(c_n)) \geq m_K(D). \qedhere
		\]
 \end{proof}

Equipped with these intermediate results, we may conclude this section.
	
	\begin{proof}[Proof of \cref{thm:main-good-recurrence}]
    Let $E \subset \N$ be such that $m_{b\Z}(\overline{\tau(E)}) = 0$ and let $\epsilon > 0$. By applying \cref{lem:empty-interior} for $K = b\Z$, there exists a compact set $F \subset b\Z$ such that $m_{b\Z}(F) > 1 - \epsilon$ and $\overline{\tau(E)} + F$ has an empty interior. \cref{lem:pw-syndetic} implies that for all $z \in b\Z$, 
		\[
		  E_z := \{n \in \N: z + \tau(n) \in \overline{\tau(E)} + F\}
		\]
		is not piecewise syndetic. 
		
		In view of \cref{lem:SameIntegral}, we can choose $z\in b\Z$ so that $A :=\{n \in \N: z - \tau(n) \in F\} \cap E$ satisfies 
		\[
		\overline{d}_{E}(A) \geq m_{b\Z}(F) > 1 - \epsilon. 
		\]
		Fixing this $z$, for any $e \in E$ and $a \in A$, we have
		\[
		z + \tau(e - a) = \tau(e) + z - \tau(a) \in \overline{\tau(E)} + F.
		\]
		Therefore, $(E - A) \cap \N \subset E_z$ and so $E - A$ is not piecewise syndetic.

  Finally, if $E$ is good for the pointwise ergodic theorem, then the upper relative density $\overline{d}_E$ in \cref{lem:SameIntegral} can be replaced by the relative density $d_E$, as a result, the same can be said about \cref{thm:main-good-recurrence}.
	\end{proof}

%\Hnote{It seems that all we need is a compact group in which $\Z$ can be mapped densely, so $b\Z$ is not the only choice.} \Anote{Interesting. We can map $\Z \to \T$ by $n \mapsto n \alpha$ for any irrational $\alpha$. However, since $\T$ is too small, $\P \alpha$ is dense. Thus we cannot use this example to show that $\P$ satisfies our hypothesis of zero Haar measure of the closure. In other words, for our applications, we need the largest group that contains a dense copy of $\Z$.}

\begin{remark}\label{remark:zero-haar-measure-again}\,
   In the proof of Proposition \ref{thm:main-good-recurrence}, all we need is a compact abelian group $K$ and a group homomorphism $\tau: \Z \rightarrow K$ such that $\tau(\Z)$ is dense in $K$, and $b\Z$ is not the only choice for $K$. See \cite[Section 3]{gll} for a construction of all such groups $K$. However, as observed by Dressler and Pigno \cite[Theorem 1]{Dressler-Pigno-Haarmeasure}, $m_K(\overline{\tau(E)})$ is minimized when $K = b\Z$ (where $m_K$ is the probability Haar measure on $K$). Therefore, the choice $K = b\Z$ is optimal in the statement of Theorem \ref{thm:main-good-recurrence}.
\end{remark}

\section{The converse: \texorpdfstring{$E$-intersectivity}{E-intersectivity} implies intersectivity}

\label{sec:converse-prime-intersective-imply-n-intersective}

The goal of this section is to prove \cref{thm:e-intersective-implies-n-intersective}. It follows from a more general theorem below regarding sets of multiple recurrence.

\begin{definition}
    Let $E$ be an infinite subset of $\N$. For $k \in \N$, a set $S \subset \N^k$ is called a \emph{$k$-intersective set for $E$} if for any $A \subset E$ such that $\overline{d}_E(A) > 0$, there exists $(n_1, \ldots, n_k) \in S$ such that
    \[
        A \cap (A - n_1) \cap \ldots \cap (A - n_k) \neq \varnothing.
    \]
\end{definition} 
For example, Szemer\'edi's theorem \cite{Szemeredi75} says that for any $k$, the set $\{ (n, 2n, \ldots, kn) : n \in \N\} \subset \N^k$ is $k$-intersective for $\N$. The polynomial Szemer\'edi theorem \cite{Bergelson-Leibman} says that if $P_1, \ldots, P_k \in \Z[x]$ are polynomials without constant term, then the set $\{ (P_1(n), \ldots, P_k(n)) : n \in \N\}$ is $k$-intersective for $\N$. 

\cref{thm:e-intersective-implies-n-intersective} corresponds to the case $k = 1$ of the next proposition.

\begin{proposition}\label{thm:recurrence_with_N}
Let $E \subset \N$ be an infinite set and $k \in \N$. Every $k$-intersective set for $E$ is a $k$-intersective set for $\N$.
\end{proposition}
\begin{proof}
 Suppose $S \subset \N^k$ is not a $k$-intersective set for $\N$. We will prove $S$ is not a $k$-intersective set for $E$. 

    Since $S$ is not a $k$-intersective set for $\N$, there exists $B \subset \N$ such that $\overline{d}(B) > 0$ and 
    \begin{equation}\label{eq:B-n-1}
        B \cap (B - n_1) \cap \ldots \cap (B - n_k) = \varnothing \text{ for all $(n_1, \ldots, n_k) \in S$.}
    \end{equation}
    By Furstenberg's Correspondence Principle \cite[Theorem 1.1]{Furstenberg77}, there exist a measure preserving system $(X, \mathcal{B}, \mu, T)$ and a set $A \subset X$ with $\mu(A) = \overline{d}(B)$ such that 
    \[
        \mu(A \cap T^{-h_1} A \cap \ldots \cap T^{-h_k} A) \leq \overline{d}(B \cap (B - h_1) \cap \ldots \cap (B - h_k)) \text{ for all $(h_1, \ldots, h_k) \in \N^k$.}
    \]
     Therefore, \eqref{eq:B-n-1} implies
		\[
		\mu(A \cap T^{-n_1} A \cap \ldots \cap T^{-n_k} A) = 0 \text{ for all } (n_1, \ldots, n_k) \in S.
		\]
		Define $A' = A \setminus \bigcup_{(n_1, \ldots, n_k) \in S} (A \cap T^{-n_1} A \cap \ldots \cap T^{-n_k} A)$. Then we have $\mu(A') = \mu(A) > 0$ and  
		\[
		A' \cap T^{-n_1} A' \cap \ldots \cap T^{-n_k} A'  = \varnothing \text{ for all $(n_1, \ldots, n_k) \in S$.}
		\]
		Therefore, by replacing $A$ with $A'$, we can assume 
		\[
		A \cap T^{-n_1} A \cap \ldots \cap T^{-n_k} A= \varnothing \text{ for all } (n_1, \ldots, n_k) \in S.
		\]
		It follows that for every $x \in X$, the set $A_x := \{n \in \N: T^n x \in A\}$ satisfies
        \[
        A_x \cap (A_x - n_1) \cap \ldots \cap (A_x - n_k) = \varnothing \text{ for all $(n_1, \ldots, n_k) \in S$.}
        \]

		Enumerate $E = \{c_1, c_2, ...\}$ in the increasing order. Let $f = 1_A \in L^{\infty}(X)$. For $N \in \N$, define the function
		\[
		f_N(x) = \frac{1}{N} \sum_{n=1}^N f(T^{c_n} x).
		\]
		Since $\mu$ is $T$-invariant, $\int_X f_N \ d\mu = \int_X f \ d \mu$. By Fatou's lemma:
		\[
		\int_X \limsup_{N \to \infty} f_N \ d \mu \geq \limsup_{N \to \infty} \int_X f_N \ d \mu = \int_X f \ d \mu = \mu(A). 
		\]
		Therefore, the set
		\[
		R := \{x \in X: \limsup_{N \to \infty} f_N (x)\geq \mu(A)\} 
		\]
		has positive measure; in particular, $R$ is non-empty.
		
		Let $x \in R$. Then
		\[
		|A_x \cap \{c_1, \ldots, c_N\}| = \sum_{n=1}^N 1_A(T^{c_n} x).
		\]
		Therefore,
		\[
		\overline{d}_E(A_x) = \limsup_{N \to \infty} \frac{|A_x \cap \{c_1, \ldots, c_N\}|}{N} = \limsup_{N \to \infty} f_N(x) \geq \mu(A) > 0.
		\]
		In other words, $B := A_x \cap E$ is a subset of $E$ of positive upper density. However,
  \[
    B \cap (B - n_1) \cap \ldots \cap (B - n_k) \subset A_x \cap (A_x - n_1) \cap \ldots \cap (A_x - n_k) = \varnothing
  \]
  for all $(n_1, \ldots, n_k) \in S$. In other words, $S$ is not a $k$-intersective set for $E$.
	\end{proof}

	\begin{remark}
	    We remark that the chromatic analogues of \cref{thm:recurrence_with_N} and \cref{thm:e-intersective-implies-n-intersective} are obvious: because for every $E\subset\N$, a partition of $\N$ automatically induce a partition of $E$, a chromatically $E$-intersective set is chromatically intersective.
     %More generally, because for $E\subset E'$, a partition of $E'$ automatically induce a partition of $E$, every chromatically $E$-intersective set is chromatically $E'$-intersective.
     %Indeed, let a set $R$ be chromatically $E$-intersective and consider a coloring of $E'=\bigcup_{i\in [k]}E'_i$. This induces a coloring of $E=\bigcup_{i \in [k]} (E\cap E'_i)$. Since $A\cap \bigcup_{i \in [k]} (E\cap E'_i-E\cap E'_i)\neq\varnothing$, it follows that $A\cap \bigcup_{i \in [k]} (E'_i-E'_i)\neq\varnothing$.
	\end{remark}
	
\section{Chromatic intersectivity versus density intersectivity}

\label{sec:chromatic_vs_density}

\subsection{For an arbitrary ambient set}

%The goal of this section is to prove \cref{thm:chromatic_E_not_density}, which says that for every infinite set $E \subset \N$, there exists a set $S$ which is chromatically intersective but not density $E$-intersective. The proof of this theorem has two ingredients: the first one shows the existence of a chromatically intersective set for $E$ which is not density intersective for $\N$; the second ingredient is \cref{thm:e-intersective-implies-n-intersective}. The first ingredient was implicitly proved in Theorem 1.2 from \cite{Griesmer-separating-topological-recurrence-measurable}.
%However, since the relative notion of intersective sets for $E$ is not introduced in \cite{Griesmer-separating-topological-recurrence-measurable}, it requires us to go through the proof of \cite[Theorem 1.2]{Griesmer-separating-topological-recurrence-measurable} to verify the following relative result which is needed for our purpose.

% The only difference is that \cite{Griesmer-separating-topological-recurrence-measurable} does not mention $E$-recurrence but merely recurrence as the relative notion is not introduced from this reference. However, the proof happens to provide $E$-intersectivity and in fact, by following the same proof, one can get the following stronger version.

\cref{thm:chromatic_E_not_density} is a corollary of the next proposition, which in turn was implicitly proved in Theorem 1.2 from \cite{Griesmer-separating-topological-recurrence-measurable}.
However, since the relative notion of $E$-intersective sets is not introduced in \cite{Griesmer-separating-topological-recurrence-measurable}, it requires us to go through the proof of \cite[Theorem 1.2]{Griesmer-separating-topological-recurrence-measurable} and verify all the steps in the relative context.
We perform this tedious task in \cref{sec:appendix_chromatic_not_density}.

%Since Theorem \ref{th:griesmer} is implicitly proven in Theorem 1.2 from \cite{Griesmer-separating-topological-recurrence-measurable}, all we have to do is go through the proof and check all the steps proceed in the relative context.

\begin{proposition}
\label{th:griesmer}
    For every infinite set $E\subset\N$ and $\delta\in (0,1/2)$, there exists a set $S\subset \N$ which is chromatically $E$-intersective
    and a set $A\subset\N$ of upper density at least $\delta$ such that $S \cap (A - A) = \varnothing$.
\end{proposition}
% \begin{proof}[Proof of \cref{thm:chromatic_E_not_density} assuming \cref{th:griesmer}]
%     Because the set $S$ found in \cref{th:griesmer} is not density intersective, by \cref{thm:e-intersective-implies-n-intersective}, $S$ is not density $E$-intersective. Thus $S$ is chromatically intersective but not density $E$-intersective. 
% \end{proof}

% In particular, $S$ is not density recurrent, hence not density $E$-recurrent by \cref{thm:e-intersective-implies-n-intersective}.

\begin{remark}
As mentioned in the introduction, the strength of \cref{th:griesmer} and \cref{thm:chromatic_E_not_density} is in their generality. However, the price for this generality is the lack of understanding of the chromatically $E$-intersective set $S$ found in these results. This is in contrast with the special case $E = \P$ where we can take $S$ to be thickly syndetic. In this remark, we discuss another difference between the general case and the case $E = \P$.

For every $\delta \in (0, 1/2)$, \cref{th:griesmer} gives a set $S$ which is chromatically $E$-intersective and a subset $A \subset \N$ such that $\overline{d}(A) > \delta$ and $S \cap (A - A) = \varnothing$ (in particular, $S$ is not intersective). To show that $S$ is not $E$-intersective, we run $S$ through the proof of \cref{thm:recurrence_with_N}. In this proof, we find $A_x \subset E$ such that $\overline{d}_E(A_x) \geq \overline{d}(A) > \delta$ and $S \cap (A_x - A_x) = \varnothing$. Thus, we only know that the upper relative density of $A_x$ is bounded below by $\delta$. On the other hand, when $E = \P$, we can produce a subset of full relative density whose difference set is disjoint from certain chromatically prime intersective set. Indeed, as a consequence of \cref{th:prime-elem}, there are a thick set $S$ and a subset $A \subset \P$ such that $d_{\P}(A) = 1$ and $S \cap (A - A) = \varnothing$. Furthermore, \cref{prop:piecewiseSyndetic}, which we will prove shortly, says that this set $S$ (more generally, every thick set) is chromatically prime intersective. 
\end{remark}

\subsection{For primes}

 \label{sec:for_primes_pwsyndetic}
	
Theorem \ref{th:prime-elem} says that there is a thick set $R$ which is not density $\P$-intersective. But we may wonder whether it is at least chromatically $\P$-intersective. Theorem \ref{prop:piecewiseSyndetic} answers this question positively and we prove it here.
In fact, we prove two different enhancements of Theorem \ref{prop:piecewiseSyndetic}: the first one is a quantitative version of Theorem \ref{prop:piecewiseSyndetic}, whereas the second one is qualitative and 
applies to a vast category of sets, beyond the set of all primes.
%\Hnote{I switch the order of these theorems since I want to define admissible sets properly first. In the second theorem, we can call them generalized admissible sets.} 

\subsubsection{A quantitative enhancement of \cref{prop:piecewiseSyndetic}}
A set $A\subset \N$ is a \textit{$\Delta_r^*$-set} if $A  \cap (S-S)\neq \varnothing$ for every subset $S$ of $\Z$ with
$|S| = r$. Any $\Delta_r^*$-set is syndetic (see footnote \ref{footnote:thick}).
%In other words, $A$ is a $\Delta_r^*$-set if $( A\setminus \{0\} +S) \cap S \neq \emptyset$ for every subset $S$ of $\Z$ with

The \textit{syndeticity index} of a set $A\subset \N$ is the smallest cardinality of a set $S \subset \Z$ such that $A + S \supset \N$. If $A$ is a $\Delta_r^*$-set, then $A$ is syndetic with index at most $r-1$. 
To see this, choose $t_1 \in \N$ arbitrary and for choose $t_2, t_3, \ldots \in \N$ recursively such that $t_k > t_{k-1}$ and $t_{k} \not \in \bigcup_{i=1}^{k-1} (t_i + A)$ for any $k \geq 2$. Since $A \cap \{t_j - t_i : 1 \leq i < j \leq k \} = \varnothing$ for all $k$, this process must stop at some $k \leq r-1$. We then have $\N \setminus \bigcup_{i=1}^{k} (t_i + A)$ is finite and let $m$ the greatest integer in this set. It follows that $\N \subset \bigcup_{i=1}^{k} (t_i - m + A)$.  %\Hnote{Can we say anything about the gap of a $\Delta_r^*$-set???} \Anote{Check this statement.}

We extend the notions of $\Delta_r^*$-set and syndeticity index to subsets of $\Z$ in the usual way, i.e. a set $A \subset \Z$ is a $\Delta_r^*$-set if $A \cap \N$ is, and the syndeticity index of $A$ is that of $A \cap \N$.

A set $H = \{ h_1, \ldots, h_k \}$ of integers is said to be \textit{admissible} if for every prime $p$, the elements of $H$ do not occupy all the residues modulo $p$. The Hardy-Littlewood conjecture says that if $H$ is admissible, then there are infinitely many $n \in \N$ such that all elements of $n+H$ are primes.
The Maynard-Tao theorem \cite{Maynard15} says that for any $r >0$, if $H$ is any admissible set with $|H| \gg r^2 e^{4r}$, there exists infinitely many $n \in \N$ such that $|(n+H) \cap \P| \geq r$.

Pintz \cite[Theorem 2]{Pintz} proved that $\P - \P$ is syndetic; however, his proof does not give a bound on the syndeticity index of $\P - \P$. Huang and Wu \cite{Huang-Wu-DifferencePrime} proved that $\P - \P$ is a $\Delta^*_{721}$-set. The next proposition extends Huang and Wu's result to any colorings of $\P$.

	\begin{proposition}
		\label{prop:piecewiseSyndetic2} For each $r \geq 1$ there is a number $m = m(r) \ll r^3 e^{4r}$ such that the following holds. 
		For any partition $\P = \bigcup_{i=1}^{r} P_i$, 
		the set $\bigcup_{i=1}^r (P_i-P_i)$ is $\Delta_m^*$. In particular, $\bigcup_{i=1}^r (P_i-P_i)$ is syndetic with index $\ll r^3 e^{4r}$.
	\end{proposition}
\cref{prop:piecewiseSyndetic2} directly implies Theorem \ref{prop:piecewiseSyndetic} since
any thick set meets every syndetic set. 
	
\begin{proof}
To begin with, we recall the following observation of Huang and Wu \cite{Huang-Wu-DifferencePrime}: 

\noindent \textbf{Claim:}
      If $A \subset \Z, |A| = m \geq k \prod_{
      \substack{p \in \P,\\ p \leq k}} (1-1/p)^{-1}$ (i.e. $k \ll \frac{m}{ \log m}$), then $A$ contains an admissible set $B$ of cardinality $k$. 

Indeed, by sieving out one residue modulo $p$, for each $p \leq k$, we have a set $B \subset A$ of cardinality $ \geq m \prod_{
\substack{p \in \P,\\ p \leq k}} (1-1/p) \geq k$ which has the property that for each $p \in \P, p \leq k$, $B$ misses at least one residue modulo $p$.  By removing additional elements from $B$, we may assume $|B| = k$. Clearly for each $p \in \P, p > k$, $B$ does not occupy all the residues modulo $p$. Hence, $B$ is admissible.

Let $A$ be any set of cardinality $m \gg r^3 e^{4r}$, then $A$ contains an admissible set $B$ of cardinality $k \gg r^2 e^{4r}$. By the Maynard-Tao theorem, there are infinitely many translates $x+B$ that contain at least $r+1$ primes. Two of these primes must have the same color, so $ (B-B)  \cap \bigcup_{i=1}^r (P_i-P_i) \neq \{ 0 \}$ and therefore $ (A-A)  \cap \bigcup_{i=1}^r (P_i-P_i) \neq \{ 0\}$. 
% The ``in particular'' follows from the fact that any thick set
% contains a set of the form $(A-A)\cap\N$, where $A\subset \N$ 
%is infinite. 
\end{proof}

\subsubsection{A generalization of \texorpdfstring{\cref{prop:piecewiseSyndetic}}{Theorem F} to other sets}
The next theorem gives a criterion on $E \subset \N$ for every thick set to be chromatically $E$-intersective. Its proof is partially inspired by Pintz \cite{Pintz}. Its statement involves a generalization of the notion of admissible sets.

\begin{proposition}
\label{th:admissible}
    Let $E$ be a set of positive integers.
    Suppose that there exists a family $\mathcal{F}$ of finite sets of positive integers,
called \textup{generalized admissible sets}, satisfying the following two properties.
\begin{enumerate}
    \item \label{th:admissible-1} For every $k\in\N$, there  exists $\ell \in \N$ such that for all 
    $F\in \mathcal{F}$ of cardinality at least $\ell$, the set $|\{n\in\N : |(n+F)\cap E|\geq k\}|$ is infinite.
    \item \label{th:admissible-2} For every $\ell\in\N$, there exists $C\in\N$ such that for every family $I_1,\ldots,I_\ell$
    of intervals of length at least $C$, there exists $F=\{f_1,\ldots,f_\ell\}\in\mathcal{F}$ such that $f_j\in I_j$ for all $j\in [\ell]$.
\end{enumerate}
Then every thick set is chromatically $E$-intersective.
\end{proposition}

\begin{proof}
    Consider a partition $E=\bigcup_{i=1}^rE_i$. Applying condition (1) with $k=r+1$, letting $\ell\in\N$
    be given by this condition, we infer that for every admissible set $F\in\mathcal{F}$ of cardinality $\ell$,
the set $X$ of $n\in\N$
    such that $\abs{(n+F)\cap E}\geq r+1$ is infinite.
    Then by pigeonholing, for every $n\in X$, there is $i\in [r]$ such that
    $\abs{(n+F)\cap E_i}\geq 2$.
    Then again by pigeonholing, there is $i\in [r]$ such that the set $Y$ of integers
    $n$ for which $\abs{(n+F)\cap E_i}\geq 2$ is infinite.

Let $T \subset \N$ be a thick set. We need to show that $T\cap \bigcup_{i\in r}(E_i-E_i)\neq\varnothing$.
By definition of $T$, there exists a sequence $(N_c)_{c\in\N}$ of integers such that
$T\supseteq \bigcup_c [N_c,N_c+c]$.
Upon extracting a sequence of $c$, we may in fact suppose
that $T\supseteq \bigcup_{k\geq 1} I_k$ where $I_k=[M_k,M_k+C_k]$ and 
$M_k>C_k>4M_{k-1}$; also $C_1>0$ may be chosen
arbitrarily large.
Let $\ell$ be given by condition (\ref{th:admissible-1})
for $k=r+1$.
Consider the intervals $I'_j=[M_j+C_j/2,M_j+C_j]$ for $j\in [\ell]$.
By condition (\ref{th:admissible-2}), there exists a generalized admissible set $F=\{f_1,\ldots,f_\ell\}\in\mathcal{F}$, where 
 $f_j\in I'_j$ for each $j\in [\ell]$, if $C_1$ is large enough. 
Thus there exists $i\in [r]$ such that $\abs{(n+F)\cap E_i}\geq 2$ for some $n\in\N$.
Therefore $f_j-f_m\subset E_i-E_i$ for some $1\leq m<j\leq \ell$.

By assumption, $0 \leq f_m \leq 2M_j \leq 2M_{j-1}$, so $f_j-f_m \in [M_j+C_j/2-2M_{j-1},M_j+C_j]\subset I_j$. We infer $I_j\cap (E_i-E_i)\neq\varnothing$, hence $T\cap (E_i-E_i)\neq\varnothing$.
\end{proof}

Many sets $E$ satisfy the hypothesis of \cref{th:admissible}. Here are some examples:
\begin{enumerate}
    \item The set of primes, with $\mathcal{F}$ being the family of all admissible sets in the usual sense.
    %where a finite set is admissible if
%    $\abs{F \pmod {p}}\neq p$ for any prime $p$.
    Indeed, property (1) is satisfied by Maynard's theorem \cite{Maynard15}.
    Property (2) is satisfied with $C=\prod_{p\leq \ell}p$ for instance, because whenever $F=\{f_1,\ldots,f_\ell\}$ is such that $(f_i,C)=1$, we have $\abs{F\mod p}<p$ for every $p\leq \ell$, and $\abs{F\mod p}\leq \ell <p$ for every $p>\ell$. Therefore, Theorem \ref{th:admissible} implies Theorem \ref{prop:piecewiseSyndetic}.
    
    More generally, still by Maynard's theorem, the set of all primes of the form $an+b$, for any given coprime integers $a$ and $b$ satisfies the hypothesis of Theorem \ref{th:admissible}.
    Since the set $E$ of all sums of two squares contain the set of all primes of the form $4n+1$, Theorem \ref{th:admissible} also applies to $E$.
    \item 
    Subsets of the primes which retain some of the equidistribution enjoyed by the primes,
    in the form of a Siegel-Walfisz and a Bombieri-Vinogradov theorem,
    also satisfy the
    hypothesis of Theorem \ref{th:admissible}, as shown by Benatar \cite{benatar} and Maynard \cite{Maynard16}, with $\mathcal{F}$ being again the family of all admissible sets in the usual sense.
    These authors mention, respectively, sets of the form
    $\P\cap g^{-1}((0,d))$ where $g\in\R[x]$ satisfies some diophantine conditions and $d$ is a positive real, and Chebotarev sets, i.e. primes with a prescribed value of the Artin symbol respectively to a Galois number field extension. See also \cite{thorner} for the latter example.
    \item Almost all sets of integers, by \cite[Lemma 5]{decompos}. Here $\mathcal{F}$ is the family of all finite subsets of $\N$.
    ``Almost all'' refers to the probability measure on (non-cofinite) sets of integers induced by the Lebesgue measure on $[0,1]$ through the bijection provided by the binary expansion of real numbers.
\end{enumerate}

\subsection{For the sets whose gaps go to infinity}
Here we prove Theorem \ref{th:squares} which demonstrates the necessity of the fact that the gaps between consecutive primes do not go to infinity in \cref{prop:piecewiseSyndetic}.  

\begin{proof}[Proof of \cref{th:squares}]
Let $R=\bigcup_{n\geq 2} I_n$ where $I_n=[f(n),f(n)+n]$, where $f(n)$ is a sufficiently quickly increasing sequence of elements of $E$. In particular, the intervals $I_n$ may be assumed to be pairwise disjoint and so $R$ is thick.
Consider $a \in E$.
We show that there exists at most one integer $b<a$	in $E$ such that $a-b\in R$.
So assume there exists such an integer $b\in E$, and let $n\in\N$ such that $a-b\in I_n$; $n$ which is unique given $a-b$. 
By hypothesis, there exists a function $g:E\rightarrow\N$
increasing to infinity such that $a-c>g(a)$, for any $c\in E$ with $c<a$, in particular for $c=b$.
Thus $a-b\in [g(a),a-1]$. If $f$ grows sufficiently quickly in terms of $g$, there is at most one $n\in \N$ such that
$[g(a),a-1]\cap I_n\neq\varnothing$. 

So there is a unique $n$ such that there exists $c<a,c\in E$	such that $a-c\in I_n$. In particular $a>f(n)$.
Since $f(n)\in E$, we infer $a\geq f(n)+g(f(n))$.
Since $a-b\leq f(n)+n$, this implies in turn that $b\geq f(g(n))-n$.

But then for any integer $c\in E,c\neq b$ we have $\abs{c-b}\geq g(b)\geq g(f(g(n))-n)>n$, if $f$ grows sufficiently quickly in terms of $g$. Thus $a-c\not\in I_n$. This means that
$a-c\not\in R$. Whence the uniqueness of $b\in E$ such that $a-b\in R$.

This uniqueness implies that there exists a two-coloring of $c:\N\rightarrow\{1,2\}$
 such that 
 \begin{equation}
 \label{eq:coloring}
 \forall \{a,b\}\subset E,\,a-b\in R\Rightarrow c(a)\neq c(b).
 \end{equation}
 This is a standard deduction in graph theory (the chromatic number is at most one plus the degeneracy) but we provide it briefly here. We construct $c(e_n)$ inductively, where $e_1<e_2<\cdots$ is the increasing sequence of the elements of $E$.
 Suppose $c(e_1),\ldots,c(e_n)$ have been constructed and satisfy
$$\forall k,\ell  \in [n],\, e_k-e_\ell\in R\Rightarrow c(e_k)\neq c(e_\ell).$$
 Then we color $e_{n+1}$.
 Since there is at most one $k<n+1$ such that $e_{n+1}-e_k\in R$, it suffices
 to take $c(e_{n+1})\neq c(e_k)$ if such a $k$ exists, and $c(e_{n+1})\in\{1,2\}$ arbitrary otherwise.

 Now the coloring $c$ induces a bipartion $E=E_1\cup E_2$ defined by $E_i=\{e \in E : c(e)=i\}$
 for $i\in [2]$. 
 In view of \eqref{eq:coloring}, we have $R\cap (E_i-E_i)=\varnothing$ for any $i\in[2]$, as desired.
\end{proof} 
 
\section{Open questions}

\label{sec:open_questions}

In this section, we present some open questions that naturally arise from our study. The first four questions involve the set of primes and the last two are about arbitrary sets of zero Banach density.

In \cref{prop:piecewiseSyndetic}, we prove that for any finite partition $\P = \bigcup_{i=1}^k E_i$, the union $\bigcup_{i-1}^k (E_i - E_i)$ is syndetic. It follows that $E_i - E_i$ is piecewise syndetic for some $i \in [k]$. It is not clear whether we can upgrade from piecewise syndeticity to syndeticity:
\begin{question}\label{ques:partition_p_syndetic}
    For any partition $\P = \bigcup_{i=1}^{k} E_i$, does there exist $i \in [k]$ such that $E_i - E_i$ is syndetic?
\end{question}
It is worth noting that the density analogue of \cref{ques:partition_p_syndetic} is false as shown in Theorems \ref{th:prime-elem} and \ref{thm:pws}.

A natural approach to answer \cref{ques:partition_p_syndetic} would be using \cref{th:prime-elem}, in which we found a partition $\P=A\cup B$ where $A-A$ 
is not syndetic and $\overline{d}_{\P}(B) = 0$. Since $B$ is a sparse subset of $\P$, it would be plausible to expect that $B - B$ is not syndetic, thus giving a negative answer to \cref{ques:partition_p_syndetic}. However, assuming the Hardy-Littlewood Conjecture (see \cref{sec:for_primes_pwsyndetic}), we can prove that $B - B$ is syndetic.

Indeed, \cref{th:prime-elem} is a special case of \cref{lm:generalisation} when $E=\P,T=\N$ and $f(x) = x^2$. First, observe that $B=\P\setminus A\supset \P\setminus C =:\P'$ where $C$ is the set defined in the proof of \cref{lm:generalisation}.
In turn, $\P'\supset \{p\in\P : p\geq \sqrt{g(2)}, (p+[g(2)-2,g(2)])\cap\P\neq\varnothing\}$. Certainly, $[g(2)-2,g(2)]$ contains an even number $m$. %\Anote{Not true since the interval $[g(1) - 1, g(1)]$ has length only $1$.}
		So $B\supset \{p\in\P: p\geq g(2), p+m\in\P\}=:\P_m $.
Letting $k \in \N$ and $H = \{0, m, 6k, m+6k\}$. Since $m$ is even, $H$ is admissible. By the Hardy-Littlewood Conjecture, there are infinitely many $n$ such that $n + H \subset \P$. Hence, there are infinitely many $n \in \P_m$ such that $n + 6k \in \P_m$. We conclude that $6 \cdot \Z \subset \P_m - \P_m$ and therefore $\P_m - \P_m$ is syndetic.

%		Now this set $\P_m$ has the property that, whenever $H\subset\N$
%		is such that $H+\{0,m\}$ is admissible, then there are infinitely many $p\in\P_m$ such that  $p+H\subset\P_m$ -- this  follows from the Hardy-Littlewood conjecture. Applying  Proposition \ref{th:admissible}, with $E = \P_m$ and $\mathcal{F} = \{ H \subset \N: H + \{ 0,m\} \textup { is admissible} \}$, whose hypotheses are satisfied by the Hardy-Littlewood conjecture, we conclude that $\P_m-\P_m$ is syndetic.

The next question is about the chromatic analogue of Theorems \ref{th:prime-elem} and \ref{thm:pws}. More precisely, it follows from these theorems that that there exists an intersective set which is not prime intersective. Equivalently, it is false that every intersective set is prime intersective. Therefore, we ask:

% Every chromatically prime intersective set is intersective.
% Moreover, theorems \ref{th:prime-elem} and \ref{thm:pws} imply that there exists an intersective set which is not prime intersective. It is natural to study the chromatic analogue of this phenomenon:

\begin{question}\label{ques:chromatic-intersective-and-prime}
    Must every chromatically intersective set be chromatically prime intersective?
\end{question}
For the density counterpart of \cref{ques:chromatic-intersective-and-prime}, we produce thick sets (and so an intersective set) which are not prime intersective. However, the same idea will not work for \cref{ques:chromatic-intersective-and-prime}: it has been shown in \cref{prop:piecewiseSyndetic} that every thick set is chromatically prime intersective.
We also remark that the converse of \cref{ques:chromatic-intersective-and-prime} is true; it is easy to see that every chromatically prime intersective set is chromatically intersective. 

In the next question, we upgrade chromatic intersectivity to density intersectivity:
\begin{question}\label{ques:intersective-prime}
    Must every intersective set be chromatically prime intersective?
\end{question}

Questions \ref{ques:chromatic-intersective-and-prime} and \ref{ques:intersective-prime} are related to the next conjecture.
It is widely believed that $\P - \P \supset 2 \cdot \Z$. If this conjecture is true, then for every intersective set $R$, $R \cap (\P - \P) \neq \varnothing$. The following conjecture seeks to prove the second clause unconditionally.
\begin{conjecture}\label{conj:intersective_intersect_p-p}
    For every intersective set $R$, we have $R \cap (\P - \P) \neq \varnothing$.
\end{conjecture}
This conjecture is equivalent to the statement that $\P - \P$ contains a set of the form $E - E$ where $d^*(E) > 0$. 
Since Pintz \cite{Pintz} shows that $\P - \P$ is syndetic, the conjecture will be true if $R$ is a thick set. On the other hand, as proved in \cref{th:prime-elem}, the conjecture will be false if we replace the full set of primes $\P$ by certain subset $A \subset \P$ of relative density $1$. Among the previous three questions/conjectures, \cref{conj:intersective_intersect_p-p} is the weakest the following sense: a positive answer to \cref{ques:chromatic-intersective-and-prime} implies a positive answer to \cref{ques:intersective-prime}, which in turn implies \cref{conj:intersective_intersect_p-p}.

In \cref{prop:ss'-z}, we show that one cannot replace $\P$ in \cref{th:prime-elem} with an arbitrary set $E$ having $d^*(E) = 0$. More precisely, there exists $E \subset \N$ such that $d^*(E) = 0$ and if $d_E(A) = 1$, then $A - A = \Z$ (in particular, syndetic). Currently, we do not know if the same is true for  \cref{thm:pws}. Likewise, the following slightly stronger statement is open:
\begin{question}\label{conj:RelativeDensityOneMinusEps}
Does there exist $E \subset \N$ such that $d^*(E) = 0$ and if $\overline{d}_E(A) > 0$, then $A - A$ is syndetic?
\end{question}
Nevertheless, Propositions \ref{lm:generalisationEps}
and \ref{th:largecm} tend to suggest that the answer is no, perhaps even with the condition $\overline{d}_E(A) > 1-\epsilon$ instead of $\overline{d}_E(A) > 0$.

Our last question aims to generalize the fact that there is an intersective set which is not prime intersective to arbitrary set of zero Banach density.

\begin{question}\label{ques:intersective-e-intersective}
    Suppose $d^*(E) = 0$. Does there exist an intersective set which is not $E$-intersective? 
\end{question}
Note that the converse of \cref{ques:intersective-e-intersective} is false according to \cref{thm:e-intersective-implies-n-intersective}.

\appendix

\section{Subsets of \texorpdfstring{$\Z$}{Z} whose closures in \texorpdfstring{$b\Z$}{bZ} have measure zero} 
 
\label{sec:closure_measure_zero}

In this appendix, we exhibit some subsets of $\Z$ whose closures in $b\Z$ have measure $0$. Recall that $b\Z$ is the Bohr compactification of $\Z$ with normalized measure $m_{b\Z}$. For a set $A \subset \Z$, let $\overline{A}$ denote the closure of $A$ in $b\Z$.
%As remarked by Kunen and Rudin \cite[p. 132]{kunen-rudin}, there are only two known ways of showing that $m_{b\Z}(\overline{A}) =0$ 
%(the other way is to show that 0 is the only limit point of $A-A$ in the Bohr topology).
We follow the same approach as Dressler and Pigno \cite{Dressler-Pigno-Haarmeasure}, which relies on the following observations.

\begin{lemma} \label{lem:mbz}\
\begin{enumerate}
  \item  For all $c \in \N, d \in \Z$, we have $m_{b\Z}(\overline{c\cdot\Z + d}) = \frac{1}{c}$.
  \item  For any sets $A, E \subset \Z$ with $E$ finite, we have $m_{b\Z}(\overline{A \cup E}) = m_{b\Z}(\overline{A 
  })$.
 \item  Suppose $c_1, \ldots, c_k \in \N$ are pairwise coprime and $A \subset \Z$ is such that, except a finite number of exceptions, every $a \in A$ misses $c'_i$ residues $\pmod{c_i}$ for each $1 \leq i \leq k$. Then $m_{b\Z}(\overline {A}) \leq \prod_{i=1}^k \left( 1-\frac{c'_i}{c_i} \right)$.  
\end{enumerate}
\end{lemma}
\begin{proof}
The first statement simply follows from the fact that $\overline{c\cdot \Z}$ is a closed subgroup of index $c$ in $b\Z$ and $m_{b\Z}$ is translation invariant. The second statement holds because $\overline{A \cup E} = \overline{A} \cup \overline{E} = \overline{A} \cup E$, and $m_{b\Z}(E) =0$. The third statement follows from the first two and the Chinese remainder theorem.
\end{proof}
	
Let $p_i$ be the $i$-th prime. Then except for $p_i$, all primes misses one residue $\pmod{p_i}$. Hence, Lemma \ref{lem:mbz} implies that for any $k$, $m_{b\Z}(\overline{\P}) \leq \prod_{i=1}^k \left( 1-\frac{1}{p_i} \right)$. Letting $k$ go to infinity, we conclude that $m_{b\Z}(\overline{\P}) = 0$.	All of our examples are obtained in this way, by taking $\{ c_i \}$ to be an appropriate sequence of moduli. More precisely, they will be dense subsets of $\P$ and $\{ p^2 : p \in \P\}$. The fact that these subsets are dense follows from various instances of Chebotarev's density theorem. We refer the reader to \cite{stevenhagen-lenstra} for the statement and an account of Chebotarev's density theorem.

Our first example is the set of values of a polynomial $P$. It was proved by Kunen and Rudin 
\cite[Theorem 5.6]{kunen-rudin} in the case $\deg P =2$ or 3.  It was also pointed out to them by David Boyd (see  footnote 1 in \cite{kunen-rudin}) that the results holds for all $\deg P \geq 2$, though no proof was given. We will now supply a proof.

\begin{proposition}
		Let $P \in \Z[x]$
		have degree $\geq 2$ and $A= \{ P(n) : n \in \Z\}$. Then  $m_{b\Z} \left( \overline{A} \right) = 0$.
	\end{proposition}
	
	\begin{proof}
Suppose $\deg P = n >1$. The polynomial $f(x,y) = P(x) + y 
		$ is clearly irreducible in $\Z[x, y]$. By the Hilbert irreducibility theorem (see e.g. \cite[Theorem 4]{vgr}), there exists infinitely many $r \in \Z$ such that $f(x, r) = P(x) + r$ is irreducible in $\Z[x]$.
Since $m_{b\Z}$ is translation invariant, we may assume that $P$ is irreducible in $\Z[x]$.

Let $K$ be a splitting field of $P$ and $X$ be the set of all roots of $P$ in $K$. Let $G = \textup{Gal}(K/\Q)$ be the Galois group of $P$. Then $G$ is a subgroup of $S_n$ that acts transitively on $X$. 
		
\noindent		\textbf{Claim:}
			$G$ has an element without fixed points.

By Burnside's lemma, the number of orbits in $X$ is $\frac{1}{|G|} \sum_{g \in G} |X^g|$,
			where $X^g$ is the set all elements of $X$ fixed by $g$. Since the action is transitive, we have $\sum_{g \in G} |X^g| = |G|$. Since $|X^e| = n > 1$, there must be some $g \in G$ such that $|X^g| = 0$, and the claim is proved.

We now recall Frobenius's theorem (see \cite[p. 11]{stevenhagen-lenstra}), which is a precursor of Chebotarev's density theorem. Let $p$ be a prime not dividing the discriminant of $P$, then in $\F_p[x]$, $P$ factors as a product of distinct irreducible polynomials of degrees $n_1, \ldots, n_k$, where $n_1 + \cdots + n_k = n$. Frobenius's theorem says that the density of primes $p$ with given decomposition pattern $(n_1, \ldots, n_k)$ exists, and is equal to $m/|G|$, where $m$ is the number of permutations $\sigma \in G$ with cycle pattern $(n_1, \ldots, n_k)$.

Let $Q$ be the set of all primes $p$ such that $p \nmid P(n)$ for all $n \in \Z$. Applying Frobenius's theorem with $n_1=1$, we see that $d_{\P}(Q) = \frac{m}{|G|}$, where $m$ is the number of permutations $\sigma \in G$ without fixed point. Hence $d_{\P}(Q) > 0$. 

Let $\{ c_i \}_{i=1}^\infty$ be all the elements of $Q$. Applying Lemma \ref{lem:mbz} with $c_i' = 1$, we have that 
$m_{b\Z}( \overline{A}) \leq \prod_{i=1}^k \left( 1-\frac{1}{c_i} \right)$, for all $k$. Letting $k$ go to infinity, we have  $m_{b\Z} \left( \overline{A} \right) = 0$, as desired.
\end{proof}
	
Our second example generalizes Dressler and Pigno's example of sums of two squares.

\begin{proposition} \label{prop:binary}
    Let $a, b, c \in \Z$ such that $D = b^2-4ac$ is not a perfect square, and let $A = \{ax^2 + bxy + cy^2: x, y \in \Z\}$. Then $m_{b\Z} \left( \overline{A} \right) =0$.
\end{proposition}

 \begin{proof}
According to \cite[Lemma 2.8]{Le-Le-compact}, we have $m_{b\Z} \left( \overline{A} \right) \leq 4|a| m_{b\Z} \left(4a \cdot \overline{A} \right) = 4|a| m_{b\Z} \left(\overline{4a \cdot A} \right)$. Since 
\[
4a(ax^2 + bxy +cy^2) = (2ax + by)^2 - Dy^2,
\]
it suffices to show that $m_{b\Z} \left( \overline{A'} \right) = 0$, where $A' = \{ z^2 - Dt^2 : z, t \in \Z\}$.
Let $p$ be a prime number such that $D$ is not a quadratic residue modulo $p$. If $p | (z^2 - Dt^2)$, 
then we must have $p|z$ and $p|t$, so $p^2 | (z^2 - Dt^2)$. Therefore, elements of $A'$ cannot be congruent to $p, 2p, \ldots, (p-1)p \pmod{p^2}$. 

Let $Q = \{ q_i \}_{i=1}^\infty$ be all the primes for which $D$ is not a quadratic residue. We claim that $\underline{d}_{\P}(Q) >0$. Indeed, since $D$ is not a perfect square, there exists a prime $p$ such that $D = p^k m$ where $k$ is an odd positive integer and $(p, m) = 1$. Consider two cases:

\noindent \textbf{Case 1}: $p = 2$ (and $m$ is odd). Let $Q' = \{q \in \P: q \equiv 5 \pmod{ 8}, q \equiv 1 \pmod{m}\}$ and let $q \in Q'$. Then $2$ is not a quadratic residue mod $q$ while $q$ is a quadratic residue of every prime divisor of $m$. By the law of quadratic reciprocity and the fact that $q \equiv 1 \pmod 4$, every prime divisor of $m$ is a quadratic residue $\pmod{q}$. Hence $m$ is a quadratic residue $\pmod{q}$ (this remains true if $m <0$, since $-1$ is a quadratic residue $\pmod{q}$). It follows that $D$ is not a quadratic residue mod $q$.

\noindent \textbf{Case 2}: $p > 2$. Let $q'$ be a quadratic nonresidue mod $p$. Let $Q' = \{q \in \P: q \equiv q' \pmod{p}, q \equiv 1 \pmod{8m}\}$ and let $q \in Q'$. Since $q \equiv 1 \pmod{8}$, $2$ is a quadratic residue $\pmod{q}$. By the law of quadratic reciprocity, every odd prime divisor of $m$ is a quadratic residue $\pmod{q}$ while $p$ is not a quadratic residue $\pmod{q}$. Also, $-1$ is a quadratic residue $\pmod{q}$. It again implies that $D$ is not a quadratic residue mod $q$.

All two sets $Q'$ defined above satisfy $Q' \subset Q$ and by Dirichlet's theorem, $d_{\P}(Q') > 0$. Therefore, we always have $\underline{d}_{\P}(Q) > 0$; our claim is proved.

%\Anote{I added details to the claim that $\overline{d}_{\P}(Q) > 0$. If we want to keep it short, we can just say ``Since $D$ is not a perfect square, by Dirichlet's theorem and the law of quadratic reciprocity, $\overline{d}_{\P}(Q) > 0$." This statement is obvious for number theorists. For people from other fields, it is believable, but not obvious.}

%Since $D$ is not a perfect square, there exists $q \in \N$ coprime to $D$ such that $D$ is not a quadratic residue mod $q$. By Dirichlet's theorem, there is a set of positive relative density of primes $p$ such that $p \equiv q \bmod D$. In particular, $d_{\P}(Q) >0$. \Anote{The last couple of sentences need to be fixed.}
%\Hnote{Is it clear?}\Pnote{Quite, but I don't see what Dirichlet's theorem has to do with it or which one you mean. The way I see it: by the law of quadratic reciprocity extended to Jacobi symbol, $\left(\frac{D}{q}\right)=-1$ if, and only if, $q$ belongs to a certain nonemptyset set of residue classes modulo $D$.} 
Applying Lemma \ref{lem:mbz} with $c_i = q_i^2$ and $c_i' = q_i-1$, we have that 
$m_{b\Z}( \overline{A'}) \leq \prod_{i=1}^k \left( 1-\frac{q_i-1}{q_i^2} \right)$, for all $k$. Letting $k$ go to infinity, we have  $m_{b\Z} \left( \overline{A'} \right) = 0$, as desired.  
\end{proof}

Our third example generalizes the second one.

\begin{proposition}\label{prop:K-algebraic}
Let $K$ be an algebraic number field of degree $n >1$. Let $\OK$ be the ring of integers of $K$ and $\{ \omega_1, \ldots, \omega_n\}$ be an integral basis of $\OK$. Let $F(x_1, \ldots, x_n) = N_{F/\Q}(x_1 \omega_1 + \cdots + x_n \omega_n)$ and $A = \{  F(x_1, \ldots, x_n) : x_1, \ldots, x_n \in \Z\}$. Then $m_{b\Z} \left( \overline{A} \right) =0$.
\end{proposition}
\begin{proof}
Similarly to the proof of Proposition \ref{prop:binary},  there exists a set of primes $Q$ with the properties that $d_\P(Q) >0$, and whenever $q | a$ where $q \in Q, a \in A$, we have $q^2 | a$. Such primes $q$ can be characterized by the residual degrees of prime ideals $\fp \subset \OK$ lying above $q$  (i.e. $\fp \cap \Z = q\Z$). This was done in detail by Glasscock \cite[Main Theorem (I)]{glasscock}, so we will just sketch the idea. 

%We briefly summarize basic notions from algebraic number theory and refer the reader to \cite[Chapter 3]{marcus}.

Recall that the norm $N(I)$ of an ideal $I \subset \OK$ is the index $[\OK:I]$, and for $x \in \OK$,  $ N(x \OK) = |N_{K/\Q}(x)|$. Any ideal $I \subset \OK$ has a unique factorization as 
\[
I = \fp_1^{e_1} \cdots \fp_k^{e_k} 
\]
where $\fp_1, \ldots, \fp_k$ are prime ideals in $\OK$.

In particular, when $q$ is a rational prime and $I = q\OK$, then 
$\fp_1, \ldots, \fp_k$ are all the prime ideals of $\OK$ lying above $q$. 
For each $i=1, \ldots, k, \OK/\fp_i$ is a finite field extension of $\Z_q$; its dimension $f_i$ is called the residual degree of $\fp_i$. In particular, $N( \fp_i ) = q^{f_i}$.

\noindent \textbf{Claim:} Let $q$ be a prime and with the property that all prime ideals of $\OK$ lying above $q$ have residual degrees $> 1$. Suppose $a \in A$ and $q | a$. Then $q^2 | a$.
\begin{proof}[Proof of Claim]
Suppose $a = N_{K/\Q}(x)$ for some $x \in \OK$. Then, factoring $x\OK$ as a product of prime ideals and taking norms, we have 
\[
|a| = N( x \OK ) = \prod_{i=1}^k N( \fp_i )^{e_i},
\]
for some prime ideals $\fp_1, \ldots, \fp_k \in \OK$. Since $q | a$, there exists $i$ such that $q \mid N( \fp_i)$ and therefore $\fp_i$ lies above $q$. Since the residual degree of $\fp_i$ is greater than 1, we have $q^2 \mid  N( \fp_i ) \mid a$, as desired.
\end{proof}
To finish the proof of \cref{prop:K-algebraic}, one has to show that the set $Q$ of primes satisfying the claim has positive density in $\P$. This follows from an application of the Chebotarev density theorem, and the fact that a finite group cannot be covered by the conjugates of a proper subgroup. We refer the reader to \cite[Lemma 4.1]{glasscock} for details. 
\end{proof}

\section{Proof of \texorpdfstring{\cref{th:griesmer}}{Proposition 6.1}}
\label{sec:appendix_chromatic_not_density}

%\Anote{Proofread this section...}

\cref{th:griesmer} is
proven by concatenating finite sets which satisfy finitary analogues of intersective property. We go through the proof to check the claimed relative reinforcement. We need a few definitions before proceeding.

A set $S\subset\N$ is \emph{$\delta$-non-intersective} if there exists $A \subset \N$ having $\overline{d}(A) > \delta$ such that $(A - A) \cap S = \varnothing$.
Given $m\in\N$ and $B \subset [m]$, we say that $(B, m)$ \emph{witnesses the $\delta$-non-intersectivity of $S$} if $|B| > \delta m$ and $B \cap (B + S) = \varnothing$, $B + S \subset [m]$, and $B + S + S \subset [m]$.

Given $k\in\N$ and $E\subset\N$, a set $S$ is called \emph{$k$-chromatically $E$-intersective} if for every partition $E=\bigcup_{i\in [k]}E_i$
we have
$S \cap \bigcup_{i \in [k]} (E_i-E_i)\neq \varnothing$.

\cref{th:griesmer} follows from the next two lemmas.
The first one is exactly Lemma 3.6 from \cite{Griesmer-separating-topological-recurrence-measurable}. Note that in \cite{Griesmer-separating-topological-recurrence-measurable}, $\delta$-non-intersective sets are called $\delta$-nonrecurrent sets and the term ``witnesses the $\delta$-non-intersectivity'' is phrased as ``witnesses the $\delta$-nonrecurrence.''
\begin{lemma}
\label{lm:threesix}
    Let $\delta,\eta\in 
    (0, 1/2 )$. Let $E, F \subset \N$ be finite sets which are $\delta$-non-intersective
and $\eta$-non-intersective, respectively. If $(A, m)$ witnesses the $\delta$-non-intersectivity of $E$, then for
all sufficiently large $l \in \N$, there exists $C \in [lm]$ with $A \subset C$ such that $(C, lm)$ witnesses
the $2\delta\eta$-non-intersectivity of $E \cup m\cdot F$. Consequently, for all sufficiently large $m$, $E \cup m\cdot F$ is $2\delta\eta$-non-intersective.
\end{lemma}

The second one is almost Lemma 3.8 from \cite{Griesmer-separating-topological-recurrence-measurable}, the difference being
that the conclusion here is relative to $E$ (hence stronger). 
\begin{lemma}
\label{lm:threeeight}
Let $k, m \in \N$ and $\delta \in (0,1/2)$. If $E \subset \Z$ is infinite, then there is a finite
$\delta$-non-intersective set $S \subset \N$ such that $m\cdot S \subset E - E$ and $m\cdot S$ is $k$-chromatically $E$-intersective.
\end{lemma}
A simple inspection of the proof of \cite{Griesmer-separating-topological-recurrence-measurable} reveals that it actually proves this stronger version. For completeness, we provide the proof 
of \cref{lm:threeeight} below.
We now show how these two lemmas yield the desired result.
\begin{proof}[Proof of Theorem \ref{th:griesmer}]
Let $\delta\in (0,1/2)$ and $E\subset\N$. We shall find a chromatically $E$-intersective set $S$ and a
set $C \subset \N$ with $\overline{d}(C) \geq \delta$ such that $(C-C)\cap S=\varnothing$.
%, meaning $S$ is not density recurrent.
To build $S$ and $C$, we will use Lemmas \ref{lm:threesix} and \ref{lm:threeeight} to find increasing sequences of sets
$(S_n),(C_n)$, and intervals $[m_1], [m_2],\ldots$, with $m_k \rightarrow\infty$, so that the
following conditions hold for all $k \in \N$.
\begin{enumerate}[label=(\roman*)]
    \item $S_k\subset E-E$ is $k$-chromatically $E$-intersective,
    \item $C_k \subset [m_k], C_k + S_k \subset [m_k], C_k + S_k + S_k \subset [m_k ]$, and $|C_k| > \delta m_k$
    \item $(C_k - C_k) \cap  S_k = \varnothing$.
\end{enumerate}
Having constructed these sets, let $C := \bigcup_{k=1}^\infty C_k$ and $S := \bigcup_{k=1}^\infty S_k$. Then $(C - C) \cap  S = \varnothing$;
otherwise for some $k \in\N$ we would have $c - c'  = s$ for some $c, c'  \in C_k$ and $s \in S_k$. Item (i) implies $S$ is chromatically
$E$-intersective and Item
(ii) implies $\overline{d}(C) \geq \delta $ (as $|C_k \cap [m_k ]|/
m_k > \delta $ for every $k$.)

To find $S_k$ and $C_k$, we first choose $S_1 = \{1\}$, $m_1 = 2$, and $C_1 = \{0\}$, so that (i)-(iii) are
trivially satisfied with $k = 1$.
To perform the inductive step, we assume the sets $S_k, C_k$, the integer $m_k$ have
been constructed and choose $\delta_k > \delta $ so that $|C_k| > \delta_km_k$. Since $\delta_k > \delta $,
we may choose $\eta < 1/2$ so that $2\delta_k\eta > \delta $. \cref{lm:threeeight} provides a finite set $S'$ which is $\eta$-non-intersective such that $m_k\cdot S'\subset E-E$ and $m_k\cdot S'$ is
$(k + 1)$-chromatically $E$-intersective. Apply Lemma \ref{lm:threesix} to find $l \geq 2$
and $C_{k+1} \subset[lm_k]$ such that $(C_{k+1}, lm_k)$ witnesses the $2\delta_k\eta$-non-intersectivity of $S_k \cup m_k\cdot S' $.
Setting $m_{k+1} = lm_k$
and $S_{k+1} = S_k \cup m_k\cdot S' \subset E-E$, we get that (i)-(iii) are satisfied with $k + 1$ in place of $k$.
\end{proof}

It remains to provide a proof of Lemma \ref{lm:threeeight}. 
To do that,
 we need some preliminaries about graph theory and chromatic numbers.

A graph $\mathcal{G}$ is a set $V$ whose elements are called \emph{vertices}, together with a set $E$ of unordered pairs of elements of $V$, called \emph{edges}. A \emph{$k$-coloring} of $\mathcal{G}$ is a function $f: V \to \{1, \ldots, k\}$. We say $f$ is \emph{proper} if $f(v_1) \neq f(v_2)$ for every edge $(v_1, v_2) \in E$. The \emph{chromatic number} of $\mathcal{G}$ is the smallest $k$ such that there is a proper $k$-coloring of $\mathcal{G}$. 

Given an abelian group $\Gamma$ and subsets $V,S \subset \Gamma$, the \emph{Cayley graph based on $V,S$}, denoted $\operatorname{Cay}(V,S)$, is the graph whose vertex set is $V$, with two vertices $x, y$ joined by an edge if $x - y \in S$ or $y - x \in S$. It follows from the definitions that $S$ is $k$-chromatically $E$-intersective if and only if the chromatic number of $\operatorname{Cay}(E,S)$ is strictly greater than $k$. If the vertex set is the ambient group, i.e. $V=\Gamma$,
we abbreviate $\operatorname{Cay}(V,S)=\operatorname{Cay}(S)$.

The next lemma is essentially Lemma 5.1 in \cite{Griesmer-separating-topological-recurrence-measurable}. However, we need a slightly stronger conclusion than the result in \cite{Griesmer-separating-topological-recurrence-measurable} in the sense that the isomorphic copy of the graph $\mathcal{G}$ found in $\operatorname{Cay}(\rho^{-1}(U) \cap (E - E))$ has vertex set $E$ (instead of $G$). The proof is the same as the proof of Lemma 5.1 in \cite{Griesmer-separating-topological-recurrence-measurable}. Since the proof is short, we include it here for completeness.

\begin{lemma}[cf. {\cite[Lemma 5.1]{Griesmer-separating-topological-recurrence-measurable}}]
\label{lem:transfer_to_discrete}
Let $G$ be a discrete abelian group, $K$ be a Hausdorff abelian group $K$, and $\rho: G \to K$ be a homomorphism. Assume $E \subset G$ is such that $\rho(E)$ is dense in $K$.
 If $U \subset K$ is open, then every finite subgraph of $\operatorname{Cay}(U)$ has an isomorphic copy in $\operatorname{Cay}(E,\rho^{-1}(U) \cap (E - E))$. 

As a consequence, if $\operatorname{Cay}(U)$ has a finite subgraph of chromatic number $\geq k$, then 
\[
    \operatorname{Cay}(E,\rho^{-1}(U) \cap (E - E))
\]
has chromatic number $\geq k$.
\end{lemma}
\begin{proof}
   To prove the first statement of the lemma it suffices to prove that if $V$ is a finite
subset of $K$, then there exist $\{g_v : v \in V \} \subset E$ such that for each $v, v' \in V$ , we have
\begin{equation}
\label{eq:fiveone}
    v - v' \in U \Rightarrow g_v - g_v' \in \rho^{-1}(U).
\end{equation} 
So let $V$ be a finite subset of $K$. Let $S := (V - V ) \cap U$, and let $W$ be a neighborhood of
$0$ in $K$ so that $S + W \subset U$ (one may take $W=\bigcap_{s\in S}(U-s)$, since $S\subset U$ is finite). Choose a neighborhood $W '$ of $0$ so that $W ' - W ' \subset W $. For
each $v \in V$ , choose $g_v \in E $ so that $\rho(g_v) \in v + W '$; this is possible since $\rho(E)$ is dense in
$K$. Also, since $K$ is Hausdorff, we can ensure that $g_v \neq g_{v'}$ if $v \neq v'$. %\Hnote{Is this true? Do we need to impose other conditions on $K$, e.g. connected?}. 
We now prove equation \eqref{eq:fiveone} holds with these $g_v$. Assuming $v - v' \in U$, we have
$\rho(g_v) - \rho(g_v' ) \in v + W ' - (v' + W ') = (v - v') + (W ' - W ') \subset v - v' + W \subset S + W \subset \subset U$,
so $g_v - g_v' \in \rho^{-1}(U)$. This proves \eqref{eq:fiveone}.
Since chromatic number is invariant under isomorphism of graphs, the second assertion
of the lemma follows immediately from the first. 
\end{proof}

For $\mathbf{x} = (x_1, \ldots, x_d) \in \F_2^d := (\Z/2\Z)^d$, we define
\[
    w(\mathbf{x}) = \left|\{1 \leq j \leq d: x_j \neq 0\}\right|.
\]
Given $1 \leq k \leq d$ and $\mathbf{y} \in \F_2^d$, the \emph{Hamming ball of radius $k$ around $\mathbf{y}$} is 
\[
    H_k(\mathbf{y}) = \{\mathbf{x} \in \F_2^d: w(\mathbf{x} - \mathbf{y}) \leq k \}.
\]
In other words, $H_k(\mathbf{y})$ is the set of $\mathbf{x} = (x_1, \ldots, x_d)$ such that $x_j \neq y_j$ for at most $k$ coordinates $j$.
We cite without proof the following result, which is a direct consequence of Lov\'asz's theorem \cite{lovasz} on the chromatic number of Kneser graphs.
\begin{lemma}[cf. {\cite[Lemma 2.5]{Griesmer-separating-topological-recurrence-measurable}}]
\label{lm:twofive}
Let $k, d \in \N$ with $2k \leq d$. Then the Cayley graph $\operatorname{Cay}(H_{2k+1}(\mathbf{1}))$
 has chromatic number at least $2k+1$. Here $\mathbf{1}$ is the vector in $\F_2^d$, all of whose components are 1.
\end{lemma}
Let $G_d = \{0, 1/2\}^d \subset \T^d$ (it is a subgroup). Note that $G_d$ is isomorphic to $(\Z/2\Z)^d$, so we will identify $H_k(\mathbf{1}) \subset (\Z/2\Z)^d$ with  the set of elements of $G_d$
which have at most $k$ coordinates that differ from $1/2$, and denote the latter by $H_k(\mathbf{1/2})$.

Let $V_{\epsilon}$ denote the open box $\{\mathbf{x} \in \T^d: \max |x_j| < \epsilon\}$. For $\mathbf{\alpha} \in \T^d$, define
\[
    \tilde{H}(\mathbf{\alpha}, k, \epsilon) = \{n \in \Z: n \mathbf{\alpha} \in H_k(\mathbf{1/2}) + V_{\epsilon}\}.
\]

\begin{lemma}[cf. Lemma 4.1(iii) in \cite{Griesmer-separating-topological-recurrence-measurable}]
\label{lm:fouroneiii}
Let $\alpha \in \T^d, 2k \leq d$ and $\epsilon > 0$.
If $E \subset \Z$ is such that $E \alpha$ is dense in $\T^d$, then there is a finite subset of $\tilde{H}(\alpha; 2k + 1, \epsilon) \cap (E - E)$ which is $2k$-chromatic $E$-intersective. 
\end{lemma}
\begin{proof}
    We use \cref{lem:transfer_to_discrete} with the open set $U = H_{2k + 1}(\mathbf{1/2}) + V_{\epsilon}$ and $\rho(n) = n \alpha$. Since $H_{2k + 1}(\mathbf{1/2}) \subset U$, $\operatorname{Cay}(U)$ contains the finite subgraph $\operatorname{Cay}(H_{2k + 1}(\mathbf{1/2}))$ (with vertex set $G_d$). Note that $\operatorname{Cay}(G_d,H_{2k + 1}(\mathbf{1/2}))$ is isomorphic to $\operatorname{Cay}(H_{2k+1}(\mathbf{1}))$, and by Lemma \ref{lm:twofive}, these graph have chromatic number $\geq 2k + 1$.

    Now \cref{lem:transfer_to_discrete} implies $\operatorname{Cay}(E,\rho^{-1}(U) \cap (E- E))$ has chromatic number $\geq 2k + 1$. Examining the definitions, we see that $\rho^{-1}(U) = \tilde{H}(\mathbf{\alpha}; 2k + 1, \epsilon)$ and so $\tilde{H}(\mathbf{\alpha}; 2k + 1, \epsilon) \cap (E - E)$ is $2k$-chromatic $E$-intersective.
\end{proof}

Before proving Lemma \ref{lm:threeeight}, we 
 state without
proof three lemmas directly drawn from \cite{Griesmer-separating-topological-recurrence-measurable}.
Given $S \subset\Z$ write
$S/m$ for the set $\{n \in\Z : mn \in S\}(= \{n/m : n \in S\}\cap\Z)$.

\begin{lemma}[cf. {\cite[Lemma 6.1]{Griesmer-separating-topological-recurrence-measurable}}]
\label{lm:sixone}
If $m \in N$ and $S \subset\Z$ is $\delta$-non-intersective, then $S/m$ is also $\delta$-non-intersective.
\end{lemma}
\begin{lemma}[cf. {\cite[Lemma 6.2]{Griesmer-separating-topological-recurrence-measurable}}]
    \label{lm:sixtwo} Let $E \subset\Z$ be an infinite set and $d \in \N$. Then there is an $\alpha \in \T^d$ such that
$E\alpha := \{n\alpha : n \in E\}$ is topologically dense in $\T^d$.
\end{lemma}
\begin{lemma}[cf. {\cite[Lemma 4.1(i)]{Griesmer-separating-topological-recurrence-measurable}}]
\label{lm:fouronei}
Let $\delta < 1/2$
and $k \in \N$. For all sufficiently large $d \in \N$ there is an $\epsilon > 0$ such that
for all $\alpha \in \T^d$, the set $\tilde{H}(\alpha; k, \epsilon)$ is $\delta$-non-intersective.
\end{lemma}

\begin{proof}[Proof of Lemma \ref{lm:threeeight}.] 
Fix $k, m \in \N$ and $\delta < 1/2$. Let $E \subset\Z$ be infinite. Let $E' \subset E$ be
an infinite subset with $m\mid (b - a)$ for all $a, b \in E'$ (i.e. the elements of $E'$ are mutually
congruent mod $m$).
By Lemma \ref{lm:fouronei}, choose $d \in \N$ large enough and $\epsilon > 0$ small enough that for
all $\alpha \in \T^d$, $\tilde{H}_\alpha := \tilde{H}(\alpha; 2k + 1, \epsilon)$ is $\delta$-non-intersective. Using \cref{lm:sixtwo}, we fix $\alpha \in \T^d$
such that $ E'\alpha$ is dense in $\T^d$. We then apply Lemma \ref{lm:fouroneiii} to find a finite
subset $S_0 \subset \tilde{H}_\alpha \cap (E' - E')$ which is $2k$-chromatically $E'$-intersective (hence $E$-intersective).
Since the elements of $E'$ are mutually congruent mod $m$, we have $S_0 \subset m\cdot\Z$, and  $S_0$ is $2k$-chromatically intersective and $\delta$-non-intersective. Let
$S := S_0/m$. By Lemma \ref{lm:sixone}, in turn $S$ is also $\delta$-non-intersective. Now $S$ is the desired set: We
have $m\cdot S = S_0 \subset E' - E' \subset E - E$, $S$ is $\delta$-non-intersective, and $m\cdot S$ is $2k$-chromatically
$E$-intersective.
\end{proof}
 
%\bibliographystyle{plain}
%\bibliography{refs}

\end{document}